\documentclass[12pt]{amsart}
\usepackage[T2A]{fontenc}
\usepackage[english]{babel}
\usepackage{amsaddr}
\usepackage{amsmath}
\usepackage{amssymb}
\usepackage{amsfonts}
\usepackage{epsfig}
\usepackage{srcltx}
\usepackage{subfigure}
\usepackage{cite}
\usepackage{float}
\usepackage[a4paper,  mag=1000, includefoot,  left=2cm, right=2cm, top=2cm, bottom=2cm, headsep=1cm, footskip=1cm ]{geometry}

\newtheorem{Th}{Theorem}

\newtheorem{Cor}{Corollary}
\begin{document}
\thispagestyle{empty}

\title[Bifurcations of autoresonant modes]
{Bifurcations of autoresonant modes in oscillating systems with combined excitation}

\author{Oskar Sultanov}

\address{Institute of Mathematics, Ufa Federal Research Centre, Russian Academy of Sciences, \newline 112, Chernyshevsky str., Ufa, Russia, 450008}
\email{oasultanov@gmail.com}


\maketitle
{\small

{\small
\begin{quote}
\noindent{\bf Abstract.}
A mathematical model describing the capture of nonlinear systems into the autoresonance by a combined parametric and external periodic slowly varying perturbation is considered. The autoresonance phenomenon is associated with solutions having an unboundedly growing amplitude and a limited phase mismatch. The paper investigates the behaviour of such solutions when the parameters of the excitation pass through bifurcation values. In particular, the stability of different autoresonant modes is analyzed and the asymptotic approximations of autoresonant solutions on asymptotically long time intervals are proposed by a modified averaging method with using the constructed Lyapunov functions.

\medskip


\end{quote}
}

\section*{Introduction}

Autoresonance is a phenomenon that occurs when a nonlinear oscillating system is perturbed by a slowly varying periodic force. Under certain initial conditions, the system automatically adjusts to the pumping and keeps this state for a sufficiently long period of time. Due to this effect, the energy of the system can increase significantly~\cite{FF01,LFS09}.
The autoresonance was first suggested in the problems associated with the acceleration of relativistic particles in the middle of the twentieth century, and nowadays it is considered as a universal phenomenon that plays the important role in a wide range of physical systems~\cite{ANSPJETP87,FGFPRE00,BFPRL14,MMFJAPRD17,BSF18}. The study of relevant mathematical models leads to new and interesting problems in the field of nonlinear dynamics~\cite{LFJPA08,LKRMS08,NVA13,KGT17}.

Mathematical models associated with the autoresonance have been actively studied recently.
In particular, the models with the external driving were investigated in~\cite{LKTMF03,LFJPA08}, while the systems with the parametric pumping were considered in~\cite{KMPRE01,KGND07}. Much less attention was paid to the systems with a combined external and parametric excitation. In particular, it was shown in~\cite{OSSIAM18} that the combined excitation leads to the coexisting of several stable autoresonant modes with different phase mismatches depending on the values of the excitation parameters. In this paper, the autoresonance model with the combined excitation is considered and the behaviour of its solutions in the vicinity of bifurcation points is investigated.

The paper is organized as follows.
In section 1, the mathematical formulation of the problem is given.
In section 2, the autoresonant modes are described and the partition of a parameter space is constructed.
In section 3, the stability of different particular autoresonant solutions is discussed.
Section 4 provides asymptotic analysis of general autoresonant solutions using the constructed Lyapunov functions.
A discussion of the results obtained is contained in section 5.

\section{Problem statement}

Consider the  non-autonomous  system of two nonlinear differential equations~\cite{OSSIAM18}:
\begin{align}
    \label{MS}
    \begin{split}
        \frac{d\rho}{d\tau}+\mu(\tau) \rho \sin (2\psi+\nu) & =  \sin\psi, \\
        \rho\Big[\frac{d\psi}{d\tau}-\rho^2+ \lambda \tau\Big]+ \mu(\tau) \rho \cos (2\psi+\nu) & = \cos\psi,
    \end{split}
\end{align}
with the parameters $\lambda>0$, $\nu\in [0,\pi)$ and a smooth given function $\mu(\tau)$.
This system arises in the study of the autoresonance phenomena for a wide class of nonlinear oscillators with the combined parametric and external chirped-frequency excitation.
The solutions $\rho(\tau)$ and $\psi(\tau)$ correspond to the amplitude and the phase mismatch of the nonlinear oscillators.
The solutions to system \eqref{MS} with $\rho(\tau)\to\infty$ and $\psi(\tau)=\mathcal O(1)$ as $\tau\to\infty$ are associated with the phase-locking and the capture into the autoresonance, while the solutions with $\rho(\tau)=\mathcal O(1)$ and $\psi(\tau)\to\infty$ as $\tau\to\infty$ relate to the phase-slipping phenomenon and the absence of the capture.

The joint influence of the parametric and the external excitations is determined by the behaviour of the function $\mu(\tau)$. Indeed, if $\mu(\tau)\sim \tau^{-1/2-\varkappa}$, $\varkappa>0$, the parametric pumping is weak and system \eqref{MS} corresponds to a perturbation of the model with the pure external excitation~\cite{LKTMF03}. If $\mu(\tau)\sim \tau^{-1/2+\varkappa}$, the external driving becomes insignificant and system \eqref{MS} represents a perturbation of the model of parametric autoresonance~\cite{OSPSIM16}.
The parametric and the external excitations act equally only when $\mu(\tau)\sim\tau^{-1/2}$. In this paper it is assumed that
\begin{gather*}
    \mu(\tau)=\tau^{-1/2}\Big(\mu_0+\sum_{k=1}^\infty \mu_k \tau^{-k}\Big),\quad \tau\to\infty, \quad \mu_k={\hbox{\rm const}}.
\end{gather*}

Note that system \eqref{MS} appears when the long-term evolution of perturbed nonlinear systems is investigated by using the averaging method~\cite{BM61}. A simple example is given by  the following equation:
\begin{gather}
\label{example}
    \frac{d^2 x}{dt^2} + (1+\epsilon \beta(t)\cos 2\zeta(t)) U'(x)=\epsilon \cos \zeta(t),
\end{gather}
where $\beta(t)=(1+\epsilon t)^{-1/2}$, $\zeta(t)=t-\vartheta t^2$, $U(x)=x^2/2- \epsilon x^4/4+\mathcal O(\epsilon^2 x^6)$, $x\to 0$, $0<\epsilon,\vartheta\ll 1$. It is easy to see that equation \eqref{example} with $\epsilon=0$ has the stable solution $x(t)\equiv 0 $.  Solutions to \eqref{example} with initial values $(x(0),x'(0))$ sufficiently close to the equilibrium $(0,0)$, whose the energy $E(t)=(x'(t))^2/2+U(x(t))$ increases significantly with time and the phase $\Psi(t)=-\arctan (x'(t)/x(t))$ is synchronized with pumping such that $\Psi(t)-\zeta(t)=\mathcal O(1)$, are associated with the capture into the autoresonance.
To approximate such solutions, introduce slow and fast variables: $\tau=\epsilon t/(2\kappa)$ and $\zeta=\zeta(t)$, where $\kappa=(4/3)^{1/3}$. Then the substitution
\begin{gather*}
x(t)=\kappa \rho(\tau) \cos(\psi(\tau)-\zeta)+\mathcal O(\epsilon)
\end{gather*}
into equation \eqref{example} and the averaging over the fast variable give system \eqref{MS} for the functions $\rho(\tau)$ and $\psi(\tau)$ with $\lambda=8\vartheta\epsilon^{-2}\kappa^2$ and $\mu(\tau)=\beta(t)\sqrt {2\kappa}/4$. A similar but more complex transition to system of type \eqref{MS} takes place when studying of autoresonant energy excitation in infinite-dimensional systems described by nonlinear partial differential equations (see, for instance,~\cite{LKRMS08}).

It follows from \cite{OSSIAM18} that system \eqref{MS} admits two or four autoresonant modes with different phase mismatches. The number of modes and their stability depend on the values of the parameters
$\mu_0$, $\lambda$, and $\nu$. It was shown that the transition of these parameters through critical values can lead to stabilization or losing the stability of the autoresonant solutions via a non-autonomous version of the centre-saddle bifurcation~\cite{Hansmann07}. The behaviour of the autoresonant solutions near the bifurcation points has not been discussed previously. In this paper the stability and the long-term asymptotics for such solutions are investigated.

In the first step, particular autoresonant solutions with unboundedly growing amplitude having power-law asymptotics at infinity are considered, and the partition of a parameter plane is outlined. Then, the stability of the particular solutions is discussed at the bifurcation points. The stability of some solutions can be investigated by trivial linear analysis failing in the study of the other ones, where nonlinear terms of equations must be considered. The stability of these solutions is analysed by constructing of suitable Lyapunov functions. The presence of stability ensures the existence of a family of autoresonant solutions with a similar asymptotic behaviour. In the last step, the asymptotic behaviour of such solutions is analyzed.

\section{Particular autoresonant solutions}

The asymptotics for the particular solutions $\rho_\ast(\tau)$, $\psi_\ast(\tau)$ with the unbounded amplitude and the bounded phase mismatch are constructed in the form of power series with constant coefficients:
\begin{gather}
    \label{PAS}
        \rho(\tau)=\rho_{-1}\tau^{1/2}+\rho_0+\sum_{k=1}^{\infty} \rho_k \tau^{-k/2}, \quad \psi(\tau)=\psi_0+\sum_{k=1}^\infty\psi_k\tau^{-k/2}, \quad \tau\to\infty.
\end{gather}
The substitution these series into system \eqref{MS} and the grouping the terms of the same power of
$\tau$ lead to finding the coefficients $\rho_{-1}= \sqrt\lambda$, $\rho_0=0$, and $\psi_0=\sigma$, where $\sigma$ is one of the roots to the following equation:
\begin{gather}
    \label{TEQ}
    \mathcal P(\sigma;\delta,\nu):= \delta\sin (2\sigma+\nu)-\sin\sigma=0,  \quad \delta:=\mu_0 \sqrt{ \lambda}.
\end{gather}
The last equation has a different number of roots depending on the values of $\delta$ and $\nu$.
If the inequality $\mathcal P'(\sigma;\delta,\nu)\neq 0$ holds, all the remaining coefficients $\rho_k$, $\psi_k$ as $k\geq 1$
are determined from the chain of linear equations:
\begin{align}
    \label{SAG}
        \begin{split}
            2 \sqrt\lambda \rho_k & = \mathcal A_k^1(\rho_{-1},\dots,\rho_{k-1}, \sigma, \psi_1,\dots,\psi_{k-1}),  \\
            \mathcal P'(\sigma;\delta,\nu)\psi_k & =\mathcal B_k^1(\rho_{-1},\dots,\rho_{k-1}, \sigma,\psi_1,\dots,\psi_{k-1}),
    \end{split}
\end{align}
where $\mathcal A_1^1=0$, $\mathcal B_1^1=-\rho_{-1}/2$,
\begin{align*}
    & \mathcal A_2^1=\frac{1}{\sqrt\lambda}\big(\delta \cos (2\sigma+\nu)-\cos\sigma\big),\\
    & \mathcal A_3^1=-\frac{\psi_1}{\sqrt\lambda} \big(2 \delta \sin(2\sigma+\nu)-\sin\sigma\big),  \\
    & \mathcal B_2^1=-\mathcal P''(\sigma;\delta,\nu)\frac{\psi_1^2}{2}-\mu_1 \rho_{-1}\sin(2\sigma+\nu), \\
    &\mathcal  B_3^1=  -\mathcal P''(\sigma;\delta,\nu)\psi_1\psi_2 - \mathcal P'''(\sigma;\delta,\nu) \frac{\psi_1^3}{6}-\mu_0 \rho_2\sin(2\sigma+\nu),
\end{align*}
etc. In particular,
\begin{gather*}
\rho_1=0, \quad \psi_1=-\theta, \quad \theta:= \frac{\sqrt \lambda}{2 \mathcal P'(\sigma;\delta,\nu)}\neq 0.
\end{gather*}
The pair of the equations $\mathcal P(\sigma;\delta,\nu)=0$ and $\mathcal P'(\sigma;\delta,\nu)=0$ defines the bifurcation curves
\begin{gather*}
    \gamma_{\pm}:=\{(\delta,\nu)\in\mathbb R\times[0,\pi): \gamma(\delta,\nu)=0, \pm\delta>0\}
\end{gather*}
in the parameter plane $(\delta,\nu)$, where $\gamma(\delta,\nu) := (4\delta^2-1)^3-27 \delta^2\sin^2\nu$. It can easily be checked that system \eqref{SAG} is solvable whenever $(\delta,\nu)\not\in\gamma_{\pm}$. The bifurcation curves divide the parameter plane into the following parts (see Fig.~\ref{PPlane}):
\begin{align*}
   & \Omega_{\pm}:=\{(\delta,\nu)\in\mathbb R\times[0,\pi): \pm \gamma(\delta,\nu)>0\}.
\end{align*}
If $(\delta,\nu)\in\Omega_{+} $, the equation $\mathcal P(\sigma;\delta,\nu)=0$ has four different roots. In the case $(\delta,\nu)\in\Omega_{-}$, there are only two different roots (see Fig.~\ref{BD23}).

\begin{figure}
\centering
\includegraphics[width=0.4\linewidth]{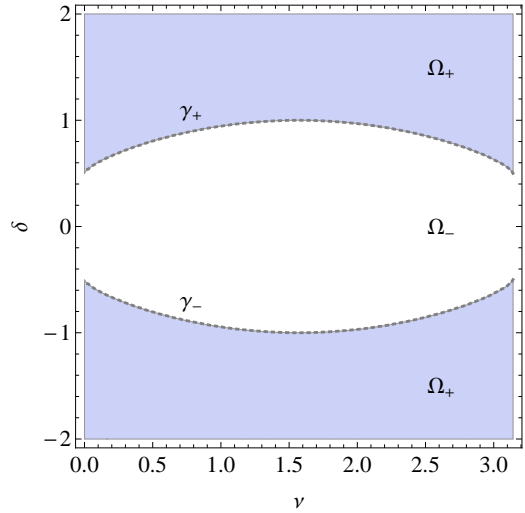}
\caption{\small Partition of the parameter plane.} \label{PPlane}
\end{figure}

\begin{figure}
\centering
\subfigure[$\nu=0$]{\includegraphics[width=0.4\linewidth]{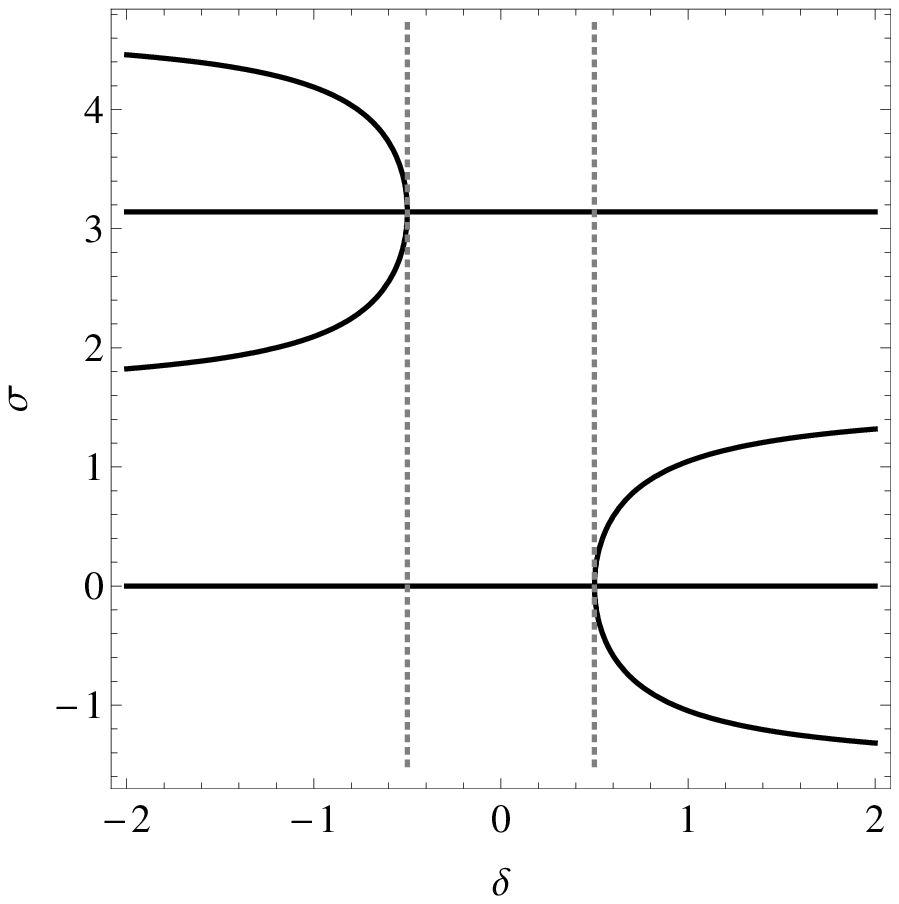}}
\hspace{4ex}
\subfigure[$\displaystyle \nu=\frac{\pi}{6}$]{\includegraphics[width=0.4\linewidth]{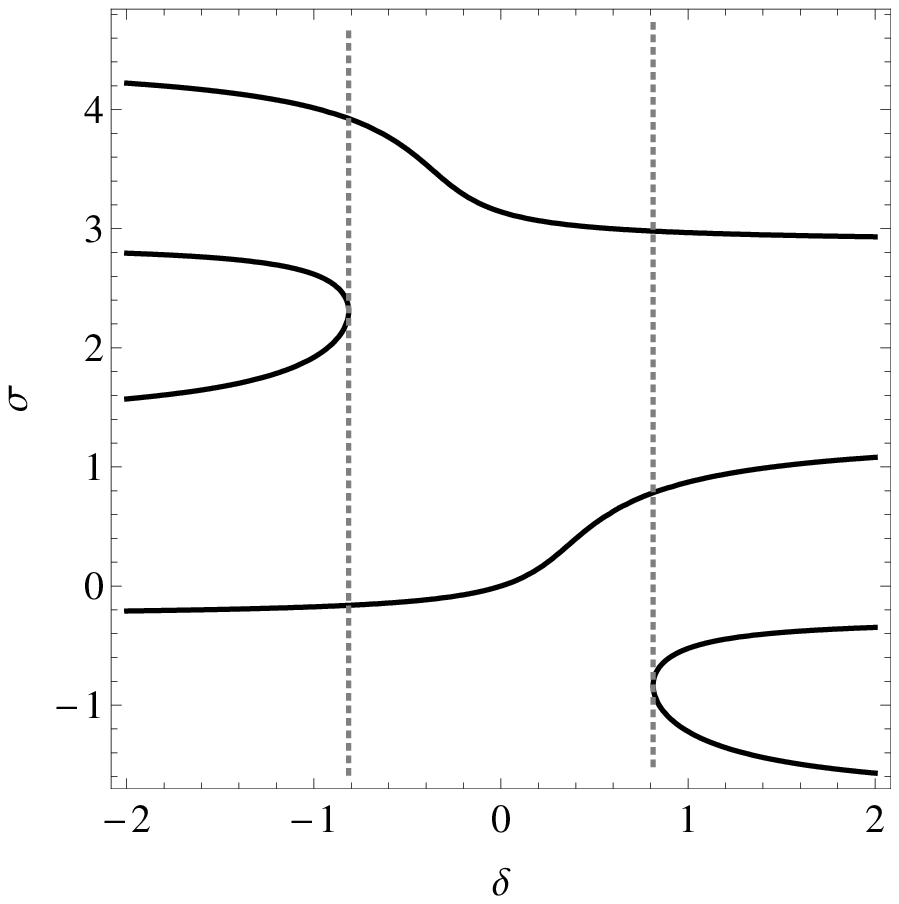}}
\caption{\small The roots to equation \eqref{TEQ} as functions of the parameter $\delta$. The vertical dotted lines correspond to $\gamma_-$ and $\gamma_+$.} \label{BD23}
\end{figure}

\begin{Th}[see~\cite{OSSIAM18}]
If $(\delta,\nu)\in\Omega_{+}$, system \eqref{MS} has four different solutions $\rho_\ast(\tau)$, $\psi_\ast(\tau)$ with asymptotic expansion in the form \eqref{PAS}.
If $(\delta,\nu)\in\Omega_{-}$, system \eqref{MS} has two different solutions $\rho_\ast(\tau)$, $\psi_\ast(\tau)$ with asymptotic expansion in the form  \eqref{PAS}.
\end{Th}

\subsection{The roots of multiplicity 2 }
If $(\delta,\nu)\in\gamma_\pm$, there exists $\sigma$ such that $\mathcal P(\sigma;\delta,\nu)=0$ and  $\mathcal P'(\sigma;\delta,\nu)=0$. In addition, suppose that $\mathcal P''(\sigma;\delta,\nu)\neq 0$, then $\sigma$ is the root of multiplicity 2. In this case, system \eqref{SAG} is unsolvable and the solution cannot be constructed in the form of \eqref{PAS}. The asymptotic solutions are sought in the form
\begin{gather}
    \label{PAS2}
        \rho(\tau)=\rho_{-1}\tau^{1/2}+\rho_0+\sum_{k=1}^{\infty} \rho_k \tau^{-k/4}, \quad \psi(\tau)=\sigma+\sum_{k=1}^\infty\psi_k\tau^{-k/4}, \quad \tau\to\infty.
\end{gather}
Substituting the series in system \eqref{MS} and equating the terms of the same power of $\tau^{1/4}$, we get $\rho_{-1}=\sqrt\lambda$, $\rho_0=\rho_1=0$ and
\begin{gather}
\label{Psi12}
\mathcal P''(\sigma;\delta,\nu) \psi_1^2+\sqrt\lambda=0.
\end{gather}
Since $\mathcal P''(\sigma;\delta,\nu)= -3\sin\sigma$, it follows that the asymptotic solution in the form \eqref{PAS2} does not exist when $\sin \sigma<0$. If $\sin\sigma>0$, equation \eqref{Psi12} has two different roots:
\begin{gather*}
   \psi_1= \pm \phi, \quad  \phi:= \sqrt{-\frac{\sqrt\lambda}{\mathcal P''(\sigma;\delta,\nu)}}.
\end{gather*}
The remaining coefficients $\rho_k$, $\psi_k$ as $k\geq 2$ are determined from the following recurrent system of equations:
\begin{align*}
             2 \sqrt\lambda \rho_{k} & = \mathcal A^{2}_{k}(\rho_{-1},\dots,\rho_{k-1}, \sigma,\psi_1,\dots,\psi_{k-1}),  \\
            \mathcal P''(\sigma;\delta,\nu)\psi_k & = \mathcal B^{2}_k(\rho_{-1},\dots,\rho_{k-1}, \sigma, \psi_1,\dots,\psi_{k-1}),
 \end{align*}
where $\mathcal A^2_2= \mathcal A^2_3 = 0$,
\begin{align*}
 & \mathcal A^2_4=\frac{1}{\sqrt\lambda}\big(\delta \cos (2\sigma+\nu)-\cos\sigma\big),\\
    & \mathcal A^2_5=-\frac{\psi_1}{\sqrt\lambda} \big(2 \delta \sin(2\sigma+\nu)-\sin\sigma\big), \\
     & \mathcal B^2_2=-\mathcal P'''(\sigma;\delta,\nu)\frac{\psi_1^2}{6} , \\
     & \mathcal B^2_3=  -\mathcal P''(\sigma;\delta,\nu) \frac{\psi_2^2}{2\psi_1} - \mathcal P'''(\sigma;\delta,\nu) \frac{\psi_1 \psi_2}{2}- \mathcal P^{(4)}(\sigma;\delta,\nu) \frac{\psi_1^3}{24}-\frac{\mu_1 \sqrt\lambda}{\psi_1}\sin(2\sigma+\nu),
\end{align*}
etc.

\subsection{The roots of multiplicity 3}

Now suppose  $(\delta,\nu)\in\gamma_\pm$ and $\sigma$ is the root of equation \eqref{TEQ} such that  $\mathcal P'(\sigma;\delta,\nu)=0$ and $\mathcal P''(\sigma;\delta,\nu)=0$. These imply that $\nu=0$, $\sin \sigma=0$ and $\mathcal P'''(\sigma;\delta,\nu)= -3\cos\sigma\neq 0$. In this case, the asymptotic solutions are constructed in the following form:
\begin{gather}
    \label{PAS3}
        \rho (\tau)=\rho_{-1}\tau^{1/2}+\rho_0+\sum_{k=1}^{\infty} \rho_k \tau^{-k/6}, \quad
        \psi (\tau)=\sigma+\sum_{k=1}^\infty\psi_k\tau^{-k/6}, \quad \tau\to\infty.
\end{gather}
It can easily be checked that
\begin{gather*}
    \rho_{-1}=\sqrt\lambda, \quad  \rho_0=\rho_1=0, \quad \psi_1 =- \chi, \quad \chi :=  \Big[\frac{3 \sqrt\lambda}{\mathcal P'''(\sigma;\delta,\nu) }\Big]^{1/3}.
\end{gather*}
The remaining coefficient $\rho_k$, $\psi_k$ as $k\geq 2$ are determined from the following chain of equations:
\begin{align*}
            2 \sqrt\lambda \rho_{k} & = \mathcal A^3_{k}(\rho_{-1},\dots,\rho_{k-1}, \sigma,\psi_1,\dots,\psi_{k-1}),  \\
            \mathcal P'''(\sigma;\delta,\nu)\psi_k & = \mathcal B^3_k(\rho_{-1},\dots,\rho_{k-1}, \sigma,\psi_1,\dots,\psi_{k-1}),
\end{align*}
where  $\mathcal A^3_2=\mathcal A^3_3=\mathcal A^3_4=\mathcal A^3_5 = 0$,
\begin{align*}
    &\mathcal A^3_6=\frac{1}{\sqrt\lambda}\big(\delta \cos 2\sigma-\cos\sigma\big), \quad \mathcal A^3_7=-\frac{\psi_1}{\sqrt\lambda}\big (2 \delta \sin2\sigma-\sin\sigma\big), \\
    & \mathcal B^3_2=-\mathcal P^{(4)}(\sigma;\delta,\nu)\frac{\psi_1^2}{12}, \quad
    \mathcal B^3_3=  - \mathcal P'''(\sigma;\delta,\nu) \frac{\psi_2^2}{\psi_1} -\mathcal P^{(4)}(\sigma;\delta,\nu)\frac{\psi_1\psi_2}{3} -\mathcal P^{(5)}(\sigma;\delta,\nu) \frac{\psi_1^3 }{60},
\end{align*}
etc.

\begin{Th} Let $(\delta,\nu)\in\gamma_{+} \cup \gamma_{-}$ and $\sigma$ be a root of equation \eqref{TEQ}.
  \begin{itemize}
  \item If $\mathcal P'(\sigma;\delta,\nu)\neq 0$, then  there exists a solution $\rho_\ast(\tau)$, $\psi_\ast(\tau)$ with asymptotic expansion in the form \eqref{PAS} with $\psi_0=\sigma$.

  \item If $\mathcal P'(\sigma;\delta,\nu)=0$ and $\mathcal P''(\sigma;\delta,\nu)<0$, then system \eqref{MS} has two different solutions $\rho_{\ast}(\tau)$, $\psi_{\ast}(\tau)$ with asymptotic expansion in the form \eqref{PAS2}.

  \item If $\mathcal P'(\sigma;\delta,\nu)=0$ and $\mathcal P''(\sigma;\delta,\nu)=0$, then system \eqref{MS} has solution $\rho_{\ast}(\tau)$, $\psi_{\ast}(\tau)$ with asymptotic expansion in the form \eqref{PAS3}.
\end{itemize}
\end{Th}
\begin{proof}
The existence of particular solutions $\rho_\ast(\tau)$, $\psi_\ast(\tau)$ with the asymptotics \eqref{PAS}, \eqref{PAS2}, or \eqref{PAS3} as $\tau\geq \tau_\ast$, $\tau_\ast={\hbox{\rm const}}>0$ follows, for instance, from~\cite {AK89,KF13}, while the comparison theorems~\cite{LKJMS14} applied to system \eqref{MS} guarantee that the solutions can be extended to the semi-axis $\tau\geq 0$.
\end{proof}

\section{Stability analysis}

\subsection{ Linear analysis} Let $\rho_\ast(\tau)$, $\psi_\ast(\tau)$ be one of the particular autoresonant solutions with asymptotics \eqref{PAS}, \eqref{PAS2} or \eqref{PAS3}.
The substitution $\rho(\tau)=\rho_\ast(\tau)+r(\tau)$, $\psi(\tau)=\psi_\ast(\tau)+p(\tau)$ into \eqref{MS} gives the following system with a fixed point at $(0,0)$:
\begin{align}
\label{ShiftSys}
\begin{split}
&\frac{dr}{d\tau}=\sin(\psi_\ast+p)-\sin\psi_\ast + \mu(\tau)\big(\rho_\ast\sin(2\psi_\ast+\nu)- (\rho_\ast+r)\sin(2\psi_\ast+2p+\nu)\big), \\
&\frac{dp}{d\tau}=2\rho_\ast r+r^2+\frac{\cos(\psi_\ast+p)}{\rho_\ast+r}-\frac{\cos\psi_\ast}{\rho_\ast}+ \mu(\tau)\big(\cos(2\psi_\ast+\nu)-\cos(2\psi_\ast+2p+\nu)\big).
\end{split}
\end{align}
Consider the linearized system:
\begin{gather*}
    \frac{d}{d\tau}\begin{pmatrix} r \\ p \end{pmatrix} = {\bf \Lambda}(\tau) \begin{pmatrix} r \\ p \end{pmatrix}, \ \
    {\bf \Lambda}(\tau):=
        \begin{pmatrix}
            \displaystyle  -\mu(\tau) \sin(2\psi_\ast+\nu) & \displaystyle \cos\psi_\ast-2  \rho_\ast \mu(\tau) \cos(2\psi_\ast+\nu) \\
            \displaystyle 2\rho_\ast - \frac{\cos\psi_\ast}{\rho_\ast^2} & \displaystyle 2 \mu  (\tau)\sin(2\psi_\ast+\nu)-\frac{\sin\psi_\ast}{\rho_\ast} \end{pmatrix}.
\end{gather*}
Define
\begin{gather*}
    \hat\rho(\tau):=    \frac{\rho_\ast(\tau)}{\sqrt{\lambda\tau}}-1, \quad   \hat\psi(\tau):=\psi_\ast(\tau)-\sigma,
\end{gather*}
for $\tau\geq 0$,
where $\sigma$ is one of the roots to equation \eqref{TEQ}. Then the functions $\hat \rho(\tau)$ and $\hat \psi(\tau)$ have the following asymptotics as $\tau\to \infty$: $\hat\rho(\tau)=\mathcal O(\tau^{-3/2})$ and
\begin{itemize}
  \item $\hat\psi(\tau)=-\theta \tau^{-1/2}+\mathcal O(\tau^{-1})$ if $\sigma$ is the simple root,
  \item $\hat\psi(\tau)=\pm \phi \tau^{-1/4}+\mathcal O(\tau^{-1/2})$ if $\sigma$ is the root of multiplicity 2,
  \item  $\hat\psi(\tau)=-\chi \tau^{-1/6}+\mathcal O(\tau^{-1/3})$ if $\sigma$ is the root of multiplicity 3.
\end{itemize}
Then the roots of the corresponding characteristic equation $|{\bf \Lambda}(\tau)-z {\bf I}|=0$ can be represented in the form
\begin{gather*}
    z_{\pm}(\tau)=\frac{1}{2}\Big( {\hbox{\rm tr}} \big( {\bf\Lambda}(\tau) \big) \pm \sqrt{{\hbox{\rm det}} \big({\bf\Lambda}(\tau)\big) }\Big),
\end{gather*}
where
\begin{gather*}
    \begin{split}
    & {\hbox{\rm tr}} \big( {\bf\Lambda}(\tau) \big) =-\frac{1}{2\tau}+ o(\tau^{-1}), \\
    & {\hbox{\rm det}}  \big({\bf\Lambda}(\tau)\big)=-8 \sqrt{\lambda\tau}\Big(\mathcal P'(\sigma;\delta,\nu)+\mathcal P''(\sigma;\delta,\nu)\hat\psi+\mathcal P'''(\sigma;\delta,\nu)\frac{\hat\psi^2}{2}+\mathcal O(\hat \psi^3)+\mathcal O(\tau^{-3/2})\Big)
    \end{split}
\end{gather*}
as $\tau\to \infty$ and $\hat\psi\to 0$.
Therefore, if $\sigma$ is the simple root of \eqref{TEQ} such that $\mathcal P'(\sigma;\delta,\nu)<0$, the leading asymptotic terms of $z_\pm(\tau)$ are real of different signs:
\begin{gather*}
z_\pm(\tau)=\pm  (4\lambda \tau)^{1/4} \sqrt{ -\mathcal P'(\sigma;\delta,\nu)}+\mathcal O(\tau^{-1/4}),\quad \tau\to\infty.
\end{gather*}
This implies that the fixed point $(0,0)$ of \eqref{ShiftSys} is a saddle in the asymptotic limit, and the corresponding solutions $\rho_\ast(\tau)$, $\psi_\ast(\tau)$ to system \eqref{MS} are unstable.
Similarly, if $\sigma$ is the root  of multiplicity 2 such that $\mathcal P''(\sigma;\delta,\nu)<0$ and $\hat\psi(\tau)=\phi\tau^{-1/4}+\mathcal O(\tau^{-1/2})$, $\phi>0$, then
\begin{gather*}
z_\pm(\tau)=\pm  (4\lambda)^{1/4} \tau^{1/8}\sqrt{-\phi\mathcal P''(\sigma;\delta,\nu)}+\mathcal O(\tau^{-1/8}),\quad \tau\to\infty.
\end{gather*}
In this case the fixed point $(0,0)$ of \eqref{ShiftSys} and the corresponding particular solution to system \eqref{MS} are both unstable.
In the same way, the fixed point of the linearized system is unstable when $\sigma$ is the root  of multiplicity 3 and $\mathcal P'''(\sigma;\delta,\nu)<0$. In this case, the eigenvalues have the following asymptotics:
\begin{gather*}
z_\pm(\tau)=\pm \lambda^{1/4} \tau ^{1/12} \sqrt{- \chi^2 \mathcal P'''(\sigma;\delta,\nu)}+\mathcal O(\tau^{-1/12}),\quad \tau\to\infty.
\end{gather*}

Thus we have
\begin{Th} Let $\sigma$ be a root of equation \eqref{TEQ}.
 \begin{itemize}
   \item If  $\mathcal P'(\sigma;\delta,\nu)<0$, the solution $\rho_\ast(\tau)$, $\psi_\ast(\tau)$ with asymptotics \eqref{PAS} is unstable.
   \item If $\mathcal P'(\sigma;\delta,\nu)=0$ and $\mathcal P''(\sigma;\delta,\nu)<0$, the solution $\rho_\ast(\tau)$, $\psi_\ast(\tau)$ with asymptotics \eqref{PAS2}, $\psi_1=\phi$ is unstable.
   \item If $\mathcal P'(\sigma;\delta,\nu)=0$, $\mathcal P''(\sigma;\delta,\nu)=0$, and $\mathcal P'''(\sigma;\delta,\nu)<0$, the solution $\rho_\ast(\tau)$, $\psi_\ast(\tau)$ with asymptotics \eqref{PAS3} is unstable.
  \end{itemize}
\label{ThExpUnst}
\end{Th}

Let us consider the following three cases that are not covered by Theorem~\ref{ThExpUnst}:
\begin{description}
  \item[Case I] $\mathcal P'(\sigma;\delta,\nu)>0$.
  \item[Case II] $\mathcal P'(\sigma;\delta,\nu)=0$, $\mathcal P''(\sigma;\delta,\nu)<0$ and $\hat\psi(\tau)=-\phi\tau^{-1/4}+\mathcal O(\tau^{-1/2})$ as $\tau\to\infty$.
  \item[Case III] $\mathcal P'(\sigma;\delta,\nu)=0$, $\mathcal P''(\sigma;\delta,\nu)=0$, and $\mathcal P'''(\sigma;\delta,\nu)>0$.
\end{description}
In these cases, the roots of the characteristic equation are complex. In particular,
\begin{gather*}
    \begin{array}{ll}
      z_\pm(\tau)=\pm i (4\lambda\tau)^{1/4} \sqrt{\mathcal P'(\sigma;\delta,\nu)}+\mathcal O(\tau^{-1/4}),&  \text{ in {\bf Case I}},\\
    z_\pm(\tau)=\pm i  (4\lambda)^{1/4} \tau^{1/8} \sqrt{-\phi \mathcal P''(\sigma;\delta,\nu)}+\mathcal O(\tau^{-1/8}),  &  \text{ in {\bf Case II}},\\
   z_\pm(\tau)=\pm i  \lambda^{1/4} \tau^{1/12}  \sqrt{ \chi^2 \mathcal P'''(\sigma;\delta,\nu)}+\mathcal O(\tau^{-1/12}),  &  \text{ in {\bf Case III}},
\end{array}
\end{gather*}
and $\Re z_\pm(\tau)=\mathcal O(\tau^{-1})$ as $\tau\to\infty$. In all three cases, the fixed point $(0,0)$ of system \eqref{ShiftSys} is a centre in the asymptotic limit. Note that the linear stability analysis fails when $\Re z_\pm(\tau)\to 0$ as $\tau\to\infty$ (see~\cite{HT94,KSDE13,OSSIAM18}), and the nonlinear terms of the equations must be taken into account.

\subsection{Lyapunov functions}

When linear stability analysis does not work, Lyapunov's method can be applied. This method assumes the existence of a locally positive definite function with a sign definite derivative along the trajectories of the system. The sign of the derivative determines whether the solution is stable or unstable. Although this method is well-developed, there is no systematic way to find appropriate Lyapunov functions. In this section, we construct such functions to prove the stability or the instability of the autoresonant solutions in different cases.

\begin{Th} Let $\sigma$ be a root of equation \eqref{TEQ}.
\label{cgs}
 \begin{itemize}
   \item In {\bf Case I}, the solution $\rho_\ast(\tau)$, $\psi_\ast(\tau)$ with asymptotics \eqref{PAS} is stable.
   \item In {\bf Case II}, the solution $\rho_\ast(\tau)$, $\psi_\ast(\tau)$ with asymptotics \eqref{PAS2}, $\psi_1=-\phi$ is unstable.
   \item In {\bf Case III}, the solution $\rho_\ast(\tau)$, $\psi_\ast(\tau)$  with asymptotics \eqref{PAS3} is unstable.
  \end{itemize}
\end{Th}
\begin{proof}
Let $\rho_\ast(\tau), \psi_\ast(\tau)$ be a one of the particular solutions with power-law asymptotics at infinity.
The change of variables
\begin{gather*}
\rho(\tau)=\rho_\ast(\tau)+ \tau^{-1/4} R(\tau),\quad \psi(\tau)=\psi_\ast(\tau)+\Psi(\tau)
\end{gather*}
transforms system \eqref{MS} into the form
\begin{gather}
    \label{ham}
    \frac{dR}{d\tau}=-\partial_\Psi H(R,\Psi,\tau), \quad
    \frac{d\Psi}{d\tau}=\partial_R H(R,\Psi,\tau)+F(R,\Psi,\tau),
\end{gather}
where
\begin{eqnarray*}
    H(R,\Psi,\tau)  & := & \tau^{-1/4} \rho_\ast  R^2  +\tau^{1/4}\big(\cos(\psi_\ast+\Psi)-\cos\psi_\ast+\Psi\sin\psi_\ast\big) \\
        &   & +\tau^{1/4} \frac{ \rho_\ast \mu}{2} \big( \cos(2\psi_\ast+\nu)-\cos(2\psi_\ast+2\Psi+\nu)-2\Psi \sin(2\psi_\ast+\nu)\big)\\
        &   & +\frac{  \mu R}{2} \big(\cos(2\psi_\ast+\nu)-\cos(2\psi_\ast+2\Psi+\nu)\big)    + \tau^{-1/2}\frac{R^3}{3}-\tau^{-1}\frac{R\Psi}{4}
\end{eqnarray*}
and
\begin{eqnarray*}
    F (R,\Psi,\tau) & := & \frac{\cos(\psi_\ast+\Psi)}{\rho_\ast+\tau^{-1/4}R}-\frac{\cos\psi_\ast}{\rho_\ast} + \frac{\mu}{2} \big(\cos(2\psi_\ast+\nu)-\cos(2\psi_\ast+2\Psi+\nu)\big)+\tau^{-1}\frac{\Psi}{4}.
\end{eqnarray*}
In the new variables $(R,\Psi)$, the system has equilibrium at the origin $(0,0)$. Note that a key role in the stability analysis is played by the asymptotic behaviour
of the right hand sides of the equations as $\tau\to\infty$ plays. Taking into account the asymptotic formulas for the particular solutions $\rho_\ast(\tau)=\sqrt{\lambda\tau} (1+\hat \rho(\tau))$ and $\psi_\ast(\tau)=\sigma+\hat\psi(\tau)$, we obtain
\begin{eqnarray*}
H & =  &\tau^{1/4}\Big( \lambda^{1/2} R^2 +  \int\limits_0^\Psi \mathcal P(\sigma+\zeta;\delta,\nu)
\,d\zeta \Big)   + \tau^{1/4} \hat\psi \Big(\mathcal P(\sigma+\Psi;\delta,\nu) - \mathcal P'(\sigma;\delta,\nu) \Psi  \Big) \\
&& + \tau^{1/4} \frac{\hat\psi^2}{2} \Big(\mathcal P'(\sigma+\Psi;\delta,\nu) - \mathcal P''(\sigma;\delta,\nu) \Psi  -\mathcal P'(\sigma;\delta,\nu)  \Big)  +\mathcal O(\tau^{1/4}\hat\psi^3 )+\mathcal O(\tau^{-1/2})
\end{eqnarray*}
and
\begin{eqnarray*}
F & =  &\tau^{-1/2}  \int\limits_0^\Psi \mathcal P(\sigma+\zeta;\delta,\nu)\,d\zeta  + \tau^{-1/2} \frac{\hat\psi}{\lambda^{1/2}}\mathcal P(\sigma+\Psi;\delta,\nu) \\
&& + \tau^{-1/2} \frac{\hat\psi^2}{2\lambda^{1/2}} \big(\mathcal P'(\sigma+\Psi;\delta,\nu) - \mathcal P'(\sigma;\delta,\nu) \big) + \tau^{-1/2} \frac{\hat\psi^3}{6\lambda^{1/2}} \big(\mathcal P''(\sigma+\Psi;\delta,\nu) - \mathcal P''(\sigma;\delta,\nu) \big) \\
&& + \tau^{-1}\frac{\Psi}{4} - \tau^{-5/4}\cos(\sigma+\Psi)\frac{R}{\lambda^{1/2}} +\mathcal O(\tau^{-1/2}\hat\psi^4)+\mathcal O(\tau^{-3/2})
\end{eqnarray*}
as $\tau\to\infty$ and for all $(R,\Psi)\in\mathcal B_{d_\ast}$, where
\begin{gather*}
    \mathcal B_{d_\ast}:=\{(R,\Psi)\in\mathbb R^2:  \sqrt{R^2+\Psi^2}\leq d_\ast\}, \quad d_\ast={\hbox{\rm const}}>0.
\end{gather*}
Note that these asymptotic formulas depend on the choice of the particular solution to system \eqref{MS} and the root of equation \eqref{TEQ}.

Consider {\bf Case I}, when $\sigma$ is the simple root to equation \eqref{TEQ} and the particular solution has asymptotics \eqref{PAS}. Hence $\hat\psi=-\theta \tau^{-1/2}+\mathcal O(\tau^{-1})$ as $\tau\to\infty$, and we get
\begin{gather*}
    H =\tau^{1/4}h_{-1}(R,\Psi)+ \tau^{-1/4} h_1(\Psi)+\mathcal O(\tau^{-1/2}), \quad
    F=\tau^{-1/2} f_2 (\Psi)+\tau^{-1} f_4(\Psi)+\mathcal O(\tau^{-5/4}),
\end{gather*}
where
\begin{gather*}
    h_{-1}:=  \lambda^{1/2} R^2 +  \int\limits_0^\Psi \mathcal P(\sigma+\zeta;\delta,\nu)\,d\zeta, \quad
        h_1 := - \theta  \big(\mathcal P(\sigma+\Psi;\delta,\nu)- \mathcal P'(\sigma;\delta,\nu) \Psi \big), \\
    f_2 :=   \int\limits_0^\Psi \mathcal P(\sigma+\zeta;\delta,\nu)\,d\zeta,\quad
    f_4 := \frac{\Psi}{4} - \frac{\theta }{\lambda^{1/2}} \mathcal P(\sigma+\Psi;\delta,\nu).
\end{gather*}
In this case, the combination
\begin{gather}
    \label{LF}
    V_1(R,\Psi,\tau):= v_{1}^0(R,\Psi,\tau)+\tau^{-3/4}v_{1}^1(R,\Psi)+\tau^{-5/4} v_{1}^2(R,\Psi)
\end{gather}
is used as a Lyapunov function candidate for system \eqref{ham}, where
\begin{gather*}
v_{1}^0=\tau^{-1/4} H(R,\Psi,\tau), \quad
v_{1}^1 = \frac{2 \lambda^{1/2} R^3}{3} + R f_2(\Psi), \quad
v_{1}^2= - \frac{ R\Psi}{8}.
\end{gather*}
Note that the functions $v_{1}^0(R,\Psi,\tau)$, $v_{1}^1(R,\Psi)$, and $v_{1}^2(R,\Psi)$ have the following asymptotics:
\begin{gather*}
v_{1}^0 =  \lambda^{1/2} R^2 +  \omega_1^2 \frac{\Psi^2}{2}+\mathcal O(d^3)+\mathcal O(\tau^{-1/2}d^2), \quad v_{1}^1=\mathcal O(d^3), \quad v_{1}^2=\mathcal O(d^2)
\end{gather*}
as $d=\sqrt{R^2+\Psi^2}\to 0$ and $\tau\to\infty$, where $\omega_1^2=\mathcal P'(\sigma;\delta,\nu)>0$.
The asymptotic estimates are uniform with respect to $(R,\Psi,\tau)$ in the domain
$\{(R,\Psi,\tau)\in\mathbb R^3: d\leq d_\ast, \tau\geq \tau_\ast\}$,
where $d_\ast$, $\tau_\ast={\hbox{\rm const}}>0$.
Then for all $0<\varkappa <1$ there exist $d_1>0$ and $\tau_1>0$ such that
\begin{gather*}
 (1-\varkappa ) w_1^2(R,\Psi)  \leq V_1(R,\Psi,\tau) \leq  (1+\varkappa ) w_1^2(R,\Psi)
\end{gather*}
for all $(R,\Psi,\tau)\in \mathcal D^{w_1}_{d_1,\tau_1}$, where
\begin{eqnarray*}
\mathcal D^{w_1}_{d_1,\tau_1}& := & \{(R,\Psi,\tau)\in\mathbb R^3: w_1(R,\Psi)\leq d_1, \tau\geq \tau_1\}, \\
w_1(R,\Psi)& :=& \sqrt{  \lambda^{1/2}R^2+\omega_1^2\frac{ \Psi^2}{2}} .
\end{eqnarray*}
Note that the combination of the functions $v_1^0$, $v_{1}^1$, and $v_{1}^2$ has the sign definite total derivative with respect to $\tau$. Indeed, it can easily be checked that
\begin{eqnarray*}
    \frac{dv_{1}^0}{d\tau}\Big|_\eqref{ham} & = & \tau^{-1/2}   \frac{\partial h_{-1}}{\partial \Psi}  f_2 + \tau^{-1} \Big(\frac{\partial  h_{-1}}{\partial \Psi} f_4 +\frac{\partial  h_{1}}{\partial \Psi} f_2 \Big)  + \mathcal O(\tau^{-5/4} d^2), \\
    &=& \tau^{-1/2}   \frac{\partial h_{-1}}{\partial \Psi}  f_2+ \tau^{-1} \Big(-\mathcal P'(\sigma;\delta,\nu)  \frac{\Psi^2}{4}   + \mathcal O(d^3)\Big)+ \mathcal O(\tau^{-5/4} d^2), \\
    \tau^{-3/4}\frac{dv_1^1}{d\tau}\Big|_\eqref{ham} & = & - \tau^{-1/2}\frac{\partial h_{-1}}{\partial \Psi}  f_2 +  \mathcal O(\tau^{-1} d^3)+\mathcal O(\tau^{-5/4} d^2), \\
    \tau^{-5/4}\frac{dv_1^2}{d\tau}\Big|_\eqref{ham} & =& \tau^{-1}  \frac{1}{4}  \Big(-  \lambda^{1/2} R^2+\mathcal \omega_1^2 \frac{\Psi^2}{2}+\mathcal O(d^3) \Big)+\mathcal O(\tau^{-5/4} d^2).
\end{eqnarray*}
Summing up the last expressions, we obtain
\begin{gather*}
\frac{dV_1}{d\tau}\Big|_\eqref{ham} = - \tau^{-1} \frac{1}{4} \Big( \lambda^{1/2} R^2 +\omega_1^2 \frac{\Psi^2}{2} +\mathcal O(d^3)\Big)+\mathcal O(\tau^{-5/4}d^2)
\end{gather*}
as $d\to 0$ and $ \tau\to\infty$.
Hence, for all $0<\varkappa <1$ there exist $d_2>0$ and $\tau_2>0$ such that
\begin{gather}
\label{V1est}
    \frac{dV_1}{d\tau}\Big|_\eqref{ham} \leq - \tau^{-1} l_\varkappa  \frac{V_1}{4 }, \quad l_\varkappa :=\frac{1-\varkappa }{1+\varkappa }
\end{gather}
 for all $(R,\Psi,\tau)\in \mathcal D^{w_1}_{d_2,\tau_2}$. Thus, for all $0<\varepsilon< d_0$ there exist $\delta_\varepsilon := \varepsilon \sqrt{l_\varkappa /2}$  such that
\begin{gather*}
\sup_{w_1\leq \delta_\varepsilon} V_1(R,\Psi,\tau)\leq (1+\varkappa )  \delta^2_\varepsilon< (1-\varkappa )   \varepsilon^2 \leq \inf_{w_1=\varepsilon}V_1(R,\Psi,\tau)
\end{gather*}
for all $\tau>\tau_0$, where $d_0=\min\{d_1,d_2\}$ and $\tau_0=\max\{\tau_1,\tau_2\}$. The last estimates and the negativity of the total derivative of the function $V_1(R,\Psi,\tau)$ ensure that any solution of system \eqref{ham} with initial data $ w_1(R(\tau_0), \Psi(\tau_0))\leq \delta_\varepsilon $ cannot leave the domain $\{(R,\Psi)\in\mathbb R^2: w_1(R,\Psi)\leq \varepsilon\}$ as  $\tau>\tau_0$. Hence, the fixed point $(0,0)$ is stable as $\tau>\tau_0$. The stability on the finite time interval $(0,\tau_0]$ follows from the theorem on the continuity of the solutions to the Cauchy problem with respect to the initial data.

Consider {\bf Case II}.
Let $\sigma$ be a root of multiplicity 2 to equation \eqref{TEQ} such that $\mathcal P'(\sigma;\delta,\nu)=0$ and $\mathcal P''(\sigma;\delta,\nu)<0$.
In this case, the particular solution $\rho_\ast(\tau)$, $\psi_\ast(\tau)$ has the asymptotics \eqref{PAS2} and  $\hat\psi=-\varphi \tau^{-1/4}+\mathcal O(\tau^{-1/2})$ as $\tau\to\infty$. Hence,
\begin{gather*}
H=\tau^{1/4} \Big(\lambda^{1/2} R^2+\mathcal P''(\sigma;\delta,\nu)\frac{\Psi^3}{6}+\mathcal O(d^4)\Big)- \phi \mathcal P''(\sigma;\delta,\nu)\frac{\Psi^2}{2}  + \mathcal O(d^3)+\mathcal O(\tau^{-1/4}d^2)
\end{gather*}
as $d\to 0$ and $\tau\to\infty$. Since the function $H(R,\Psi,\tau)$ is sign indefinite, the combination \eqref{LF} can not be used as a Lyapunov function candidate.

Consider the change of variables
\begin{gather*}
 R(\tau)=\tau^{-3/8} r(s), \quad \Psi(\tau)= \tau^{-1/4} p(s), \quad s=\frac{8}{9}\tau^{9/8}
\end{gather*}
in system \eqref{ham}. The transformed system is
\begin{gather}
    \label{ham1}
    \frac{d r}{ds}=-\partial_{p} \mathcal H_{2}(r,p,s), \quad
    \frac{d p}{ds}=\partial_{r} \mathcal H_{2}(r,p,s)+\mathcal F_{2}(r,p,s),
\end{gather}
where
\begin{gather*}
    \mathcal H_{2}(r,p,s):=\tau^{1/2} H ( \tau^{-3/8}r,  \tau^{-1/4}p, \tau ) \Big|_{\tau=(9s/8)^{8/9}}-s^{-1}\frac{rp}{3}, \\
    \mathcal F_{2}(r,p,s):=s^{-1}\frac{5 p }{9} +\tau^{1/8}  F ( \tau^{-3/8}r,  \tau^{-1/4}p, \tau ) \Big|_{\tau=(9s/8)^{8/9}}.
 \end{gather*}
Taking into account \eqref{PAS2}, it can easily be checked that
\begin{gather*}
\mathcal H_2(r,p,s) = h_{2}^0(r,p)+ \sum_{k=2}^\infty  s^{-k/9} h_{2}^k(r,p), \quad
\mathcal F_2(r,p,s)= s^{-1}\sum_{k=0}^\infty  s^{-k/9} f_2^k(r,p)
\end{gather*}
as $s\to \infty$ and for all $(r,p)\in\mathcal B_{d_\ast}$, where
\begin{gather*}
h_{2}^0(r,p)=\lambda^{1/2} r^2 +  \omega_2^2 \frac{p^2}{2}+\mathcal P''(\sigma;\delta,\nu)\frac{p^3}{6}, \quad
f_{2}^0(r,p)=\frac{p}{3} + \omega_2^2 \frac{4 p^2}{9 \lambda^{1/2}} + \mathcal P''(\sigma;\delta,\nu)\frac{ 4 p^3}{27},
\end{gather*}
$h_2^k(r,p)=\mathcal O(\Delta^2)$, $f_2^k(r,p)=\mathcal O(\Delta)$ as $ \Delta=\sqrt{r^2 +  p^2}\to 0$, and $\omega_2^2=-\phi \mathcal P''(\sigma;\delta,\nu)>0$. We see that the function $\mathcal H_2(r,p,s)$ is suitable for the basis of a Lyapunov function candidate. Consider the following combination:
\begin{gather}
\label{UF1}
U_2(r,p,s) = \mathcal H_2(r,p,s)+s^{-1}\frac{rp}{6}.
\end{gather}
It is easy to prove that for all $0<\varkappa <1$ there exist $ d_1>0$ and $ s_1>0$ such that
\begin{gather*}
   (1-\varkappa )  w_2^2(r,p) \leq U_2(r,p,s) \leq   (1+\varkappa )  w_2^2(r,p)
\end{gather*}
for all $(r,p,s)\in \mathcal D^{w_2}_{d_1,s_1}$, where
\begin{eqnarray*}
w_2(r,p) & := & \sqrt{\lambda^{1/2}r^2+\omega_2^2\frac{p^2}{2}}.
\end{eqnarray*}
It turns out that the total derivative of the function $U_2(r,p,s)$ has a sign definite leading term of the asymptotics:
\begin{align}
\begin{split}
\label{dU2}
\frac{d U_2}{ds}\Big|_{\eqref{ham1}} & =  \frac{\partial  U_2}{\partial s} + \frac{\partial \mathcal H_{2}}{\partial p} \mathcal F_{2}+s^{-1}\frac{1}{6}
\Big(r \frac{\partial \mathcal H_{2}}{\partial r} - p\frac{\partial \mathcal H_{2}}{\partial p}\Big)\\
& = s^{-1}  \frac{1}{3}\Big(\lambda^{1/2}r^2+\omega_2^2\frac{p^2}{2}+ \mathcal O(w_2^3)\Big)+\mathcal O(s^{-11/9} w_2^2)
\end{split}
\end{align}
as $\tau\to\infty$ and $w_2\to 0$. Hence, for all $0<\varkappa <1$ there exist $d_2>0$ and $s_2>0$ such that
\begin{gather}
    \label{LFD1}
    \frac{d U_2}{ds}\Big|_\eqref{ham1} \geq s^{-1 }   l_\varkappa  \frac{U_2 }{3 } \geq 0
\end{gather}
for all $(r_2,p_2,s)\in\mathcal D^{w_2}_{d_2,s_2}$.
The last inequality implies the instability of the equilibrium $(0,0)$ to system \eqref{ham1}.
Indeed, let $r(s), p(s)$ be a solution to system \eqref{ham1} with initial data $r(s_0)$ and $p(s_0)$ such that $w_2\big(r(s_0),p(s_0)\big)=\delta$, where $\delta\in (0,d_0)$, $d_0=\min\{d_1,d_2\}$, and $s_0=\max\{s_1,s_2\}$. By integrating \eqref{LFD1} with respect to $s$, we obtain the following inequality:
\begin{gather*}
w_2^2\big(r(s),p(s)\big)\geq \delta^2   l_\varkappa   \Big(\frac{s}{s_0}\Big)^{l_\varkappa/3},
\end{gather*}
as $s\geq s_0$. Hence, there exists $S_0>s_0$ such that the solution $r(s)$, $p(s)$ escapes from the domain $\{(r,p)\in \mathbb R^2: w_2(r,p)\leq d_0\}$ as $s\geq S_0$. This means that the equilibrium $(0,0)$ is unstable. Returning to the original variables $(\rho,\psi,\tau)$, we obtain the result of the theorem.

Finally, consider {\bf Case III}.  Let $\sigma$ be a root of multiplicity 3 to equation \eqref{TEQ} such that $\mathcal P'(\sigma;\delta,\nu)=0$ and $\mathcal P''(\sigma;\delta,\nu)=0$.
Then the solution $\rho_\ast(\tau)$, $\psi_\ast(\tau)$ has the asymptotics \eqref{PAS3}, and $\hat\psi=-\chi \tau^{-1/6}+\mathcal O(\tau^{-1/3})$.
In this case,
\begin{eqnarray*}
H(R,\Psi,\tau)&=&\tau^{1/4} \Big(\lambda^{1/2} R^2+\mathcal P'''(\sigma;\delta,\nu)\frac{\Psi^4}{24} +\mathcal O(d^5)\Big)+ \mathcal O(\tau^{1/12}  d^3) +\mathcal O(\tau^{-1/12} d^2)
\end{eqnarray*}
as $d\to 0$ and $\tau\to\infty$. Note that this function can not be used in the construction of a Laypunov function. Indeed, if $\Psi$ is small enough and $\tau$ is big enough (for example, $\Psi\sim \epsilon^{1/6}$ and  $\tau\sim \epsilon^{-1}$, where $0<\epsilon\ll 1$), the leading and the remainder terms in the last expression can be of the same order. Hence the function $H(R,\Psi,\tau)$ is sign indefinite in a neighborhood of the equilibrium.

Consider the change of variables
\begin{gather*}
 R(\tau)=\tau^{-1/3} r(s), \quad \Psi(\tau)= \tau^{-1/6} p(s), \quad s=\frac{12}{13}\tau^{13/12}
\end{gather*}
in system \eqref{ham}. It can easily be checked that the transformed system has the form
\begin{gather}
    \label{ham2}
    \frac{d r}{ds}=-\partial_{p} \mathcal H_3(r,p,s), \quad
    \frac{d p}{ds}=\partial_{r} \mathcal H_3(r,p,s)+\mathcal F_3(r,p,s),
\end{gather}
where
\begin{gather*}
\mathcal H_3(r,p,s):=\tau^{5/12} H ( \tau^{-1/3}r,  \tau^{-1/6}p, \tau )\Big|_{\tau=(13s/12)^{12/13}}-s^{-1}\frac{4rp}{13}, \\
\mathcal F_3(r,p,s):=s^{-1}\frac{6 p }{13} +\tau^{1/12}  F ( \tau^{-1/3}r,  \tau^{-1/6}p, \tau )\Big|_{\tau=(13s/12)^{12/13}}.
 \end{gather*}
Using asymptotic formulas for the particular solution, we obtain
 \begin{eqnarray*}
\mathcal H_3 &= &   \lambda^{1/2} r^2 +  \omega_2^2 \frac{p^2}{2}-\chi \mathcal P'''(\sigma;\delta,\nu)\frac{p^3}{6}  +\mathcal P'''(\sigma;\delta,\nu)\frac{p^4}{24}+\mathcal O(s^{-2/13}),  \\
\mathcal F_3 & = & s^{-1} \Big(\frac{3p}{13} + \mathcal P'''(\sigma;\delta,\nu) \frac{p^2}{26\lambda^{1/2}} \big(6 \chi^2  - 4 \chi p +\lambda^{1/2}p^2 \big) \Big) +\mathcal O(s^{-14/13})
\end{eqnarray*}
as $s\to\infty$ and for all $(r,p)\in\mathcal B_{d_\ast}$, where $\omega_3^2:=\chi^2 \mathcal P'''(\sigma;\delta,\nu)/2>0$.
Consider the combination
\begin{gather*}
U_3(r,p,s) = \mathcal H_3(r,p,s)+s^{-1}\frac{3 rp}{26}
\end{gather*}
as a Lyapunov function candidate for system \eqref{ham2}. It follows easily that for all $0<\varkappa <1$ there exist $ d_1>0$ and $ s_1>0$ such that
\begin{gather*}
   (1-\varkappa )  w_3^2(r,p) \leq U_3(r,p,s) \leq   (1+\varkappa )  w_3^2(r,p)
\end{gather*}
for all $(r_3,p_3,s)\in \mathcal D^{w_3}_{d_1,s_1}$, where
\begin{eqnarray*}
w_3(r,p) & := & \sqrt{\lambda^{1/2}r^2+\omega_3^2\frac{p^2}{2}}.
\end{eqnarray*}
The derivative of this function with respect to $s$ along the trajectories of system \eqref{ham2} has the following asymptotics:
\begin{eqnarray*}
\frac{d U_3}{ds}\Big|_{\eqref{ham2}} & = & \frac{\partial  U_3}{\partial s} + \frac{\partial \mathcal H_3}{\partial p} \mathcal F_3+s^{-1}\frac{3}{26}
\Big(r\frac{\partial \mathcal H_3}{\partial r} - p\frac{\partial \mathcal H_3}{\partial p}\Big)\\
& =& s^{-1}  \frac{3}{13}\Big(\lambda^{1/2}r^2+\omega_3^2\frac{p^2}{2}+ \mathcal O(w_3^3)\Big)+\mathcal O(s^{-14/13} w_3^2)
\end{eqnarray*}
as $s\to\infty$ and $w_3\to 0$. Hence, for all $0<\varkappa <1$ there exist $d_2>0$ and $s_2>0$ such that
\begin{gather}
\label{LFD2}\frac{d U_3}{ds}\Big|_\eqref{ham2} \geq s^{-1 }   l_\varkappa  \frac{3 }{13 } U_3\geq 0
\end{gather}
for all $(r,p,s)\in\mathcal D^{w_2}_{d_2,s_2}$. The last inequality implies that the fixed point $(0,0)$ of system \eqref{ham2} is unstable.
Indeed, let $r(s)$, $p(s)$ be a solution to system \eqref{ham2} with initial data $r(s_0)$, $p(s_0)$ such that $w_3\big(r(s_0),p(s_0)\big)=\delta$, where $\delta\in (0,d_0)$, $d_0=\min\{d_1,d_2\}$, and $s_0=\max\{s_1,s_2\}$. Integrating \eqref{LFD2} with respect to $s$ yields
\begin{gather*}
w_3^2\big(r(s),p(s)\big)\geq  \delta^2  l_\varkappa  \Big(\frac{s}{s_0}\Big)^{3 l_\varkappa/13},
\end{gather*}
as $s\geq s_0$. Hence there exists $S_0>s_0$ such that the solution $r(s)$, $p(s)$ escapes from the domain $\{(r,p)\in \mathbb R^2: w_3(r,p)\leq d_0\}$ as $s\geq S_0$. Thus, the fixed point $(0,0)$ of system \eqref{ham} and the particular solution $\rho_\ast(\tau)$, $\psi_\ast(\tau)$ to system \eqref{MS} are unstable.
\end{proof}

Thus, the particular autoresonant solutions with power-law asymptotics are unstable (with respect to all the variables) in {\bf Case II} and {\bf Case III}. However, due to weak instability, it can be shown that there is a partial stability~\cite{VorRum} with respect to one of the variables $(\rho,\psi)$ on an asymptotically long time interval. We have the following.

\begin{Th}
\label{c2s}
Let $\sigma$ be a root of equation \eqref{TEQ}. In {\bf Case II}, the solution $\rho_\ast(\tau)$, $\psi_\ast(\tau)$ with asymptotics \eqref{PAS2}, $\psi_1=-\phi$ is $\rho$-stable on a finite but asymptotically long time interval.
\end{Th}
\begin{proof}
Consider the change of variables
\begin{gather}
\label{chv2}
R(\tau)=\varrho(\tau), \quad \Psi(\tau)=\tau^{-1/4}\varphi(\tau)
\end{gather}
in system \eqref{ham}. It is clear that the system for the new variables $(\varrho ,\varphi )$ has the following form:
\begin{gather}
\label{hatham1}
    \frac{d\varrho}{d\tau}=-\partial_{\varphi } \hat H_2(\varrho ,\varphi ,\tau), \quad
    \frac{d\varphi }{d\tau}=\partial_{\varrho } \hat H_2(\varrho ,\varphi ,\tau)+\hat F_2(\varrho ,\varphi ,\tau),
\end{gather}
where
\begin{gather*}
\hat H_2(\varrho ,\varphi ,\tau) :=  \tau^{1/4} H(\varrho ,\tau^{-1/4}\varphi ,\tau), \quad
\hat    F_2(\varrho ,\varphi ,\tau):=  \tau^{1/4} F (\varrho ,\tau^{-1/4}\varphi ,\tau) + \tau^{-1}\frac{\varphi }{4}.
\end{gather*}
Furthermore, using asymptotic formulas for the solution $\rho_\ast(\tau)$, $\psi_\ast(\tau)$, we get
\begin{eqnarray*}
    \hat H_2(\varrho ,\varphi ,\tau)&= &  \tau^{1/2} \lambda^{1/2} \varrho ^2 + \tau^{-1/4} \Big(\omega_2^2 \frac{\varphi ^2}{2}+\mathcal P''(\sigma;\delta,\nu)\frac{\varphi ^3}{6}+\frac{\varrho ^3}{3}\Big) +\mathcal O(\tau^{-1/2}),  \\
\hat F_2(\varrho ,\varphi ,\tau)&= & \tau^{-1} \Big(-\cos \sigma \frac{\varrho }{\lambda^{1/2}} + \omega_2^2 \frac{\varphi  ^2}{2\lambda^{1/2}} + \mathcal P''(\sigma;\delta,\nu)\frac{\varphi ^3}{6} \Big) +\mathcal O(\tau^{-5/4})
\end{eqnarray*}
as $\tau\to\infty$ and for all $(\varrho ,\varphi )\in\mathcal B_{d_\ast}$. Note that there exists $\tau_\ast>0$ such that for every $\tau>\tau_\ast$, the Hamiltonian $\hat H_2(\varrho ,\varphi ,\tau)$ is a positive definite quadratic form in the vicinity of the fixed point $(0,0)$.
Consider the following perturbation of $\hat H_2(\varrho ,\varphi ,\tau)$ as a Lyapunov function candidate for system \eqref{hatham1}:
\begin{gather*}
V_2(\varrho ,\varphi ,\tau):= v_{2}^0(\varrho,\varphi,\tau)+\tau^{-3/2} v_{2}^1(\varrho,\varphi),
\end{gather*}
where
\begin{gather*}
    v_{2}^0(\varrho ,\varphi ,\tau):=\tau^{-1/2}\hat H_2(\varrho,\varphi,\tau), \quad
    v_{2}^1(\varrho,\varphi):= -\frac{3 \varrho \varphi  }{16}-\cos\sigma \frac{\varrho^2}{2\lambda^{1/2}}.
\end{gather*}
Note that for all $0<\varkappa <1$ there exist $d_1>0$ and $\tau_1\geq \tau_\ast$ such that
\begin{gather*}
 (1-\varkappa ) W_2^2(\varrho ,\varphi ,\tau)  \leq V_2(\varrho ,\varphi ,\tau)\leq (1+\varkappa ) W_2^2(\varrho ,\varphi ,\tau)
\end{gather*}
for all $(\varrho,\varphi)\in \mathcal B_{d_1}$ and $\tau\geq \tau_1$, where
\begin{gather*}
W_2(\varrho ,\varphi ,\tau)  :=  \sqrt{\lambda^{1/2}\varrho ^2+\tau^{-3/4}\omega_2^2\frac{\varphi ^2}{2}}.
\end{gather*}
The derivatives of functions $v_{2}^0(\varrho ,\varphi ,\tau)$ and $v_{2}^1(\varrho ,\varphi )$ with respect to $\tau$ along the trajectories of system \eqref{hatham1} have the following asymptotics:
\begin{eqnarray*}
    \frac{dv_{2}^0}{d\tau}\Big|_\eqref{hatham1}
        & = &
            \tau^{-1/2}   \Big(\frac{\partial \hat H_2}{\partial \tau}+ \frac{\partial \hat  H_2}{\partial \varphi }  \hat F_2\Big) -\frac{1}{2} \tau^{-3/2} \hat H_2  \\
        &=&
            \tau^{-7/4}  \Big( - \omega_2^2 \frac{3\varphi ^2}{8}-\omega_1^2 \cos\sigma\frac{\varrho \varphi  }{\lambda^{1/2}}+\mathcal O(\hat  d ^3)\Big)+ \mathcal O(\tau^{-2}\hat   d^2), \\
    \tau^{-3/2}\frac{dv_{2}^1}{d\tau}\Big|_\eqref{hatham1}
        & = &
            -\tau^{-1} \lambda^{1/2}\frac{3 \varrho ^2}{8}  + \tau^{-7/4}\Big( \omega_2^2 \frac{3\varphi ^2}{16}+\omega_2^2 \cos\sigma\frac{\varrho \varphi  }{\lambda^{1/2}}+\mathcal O(\hat  d^3)\Big) +\mathcal O(\tau^{-2} \hat  d^2)
\end{eqnarray*}
as $\hat  d=\sqrt{\varrho ^2+\varphi ^2}\to 0$ and $\tau\to\infty$. Combining the last expressions, we see that
\begin{gather*}
\frac{dV_2}{d\tau}\Big|_\eqref{hatham1} = - \tau^{-1}  \frac{3}{8}\Big( \lambda^{1/2}    \varrho ^2 +\tau^{-3/4}\omega_2^2 \frac{\varphi ^2}{2} +\mathcal O(\tau^{-3/4}\hat d^3)\Big)+\mathcal O(\tau^{-2} \hat d^2).
\end{gather*}
Hence, for all $0<\varkappa <1$ there exist $d_2>0$ and $\tau_2\geq \tau_\ast$ such that
\begin{gather}
\label{v2est}
\frac{dV_2}{d\tau}\Big|_\eqref{hatham1} \leq -\tau^{-1 } l_\varkappa   \frac{3}{8 } V_2
\end{gather}
for all $(\varrho,\varphi)\in\mathcal B_{d_2}$ and $\tau\geq \tau_2$.
Besides, for all $\varepsilon>0$ there exist $\delta_\varepsilon = \tau_0^{-3/8} \varepsilon \sqrt{l_\varkappa m_-/(2m_+)}$  such that
\begin{gather*}
\sup_{\{ \varrho ^2+\varphi ^2 \leq \delta_\varepsilon^2\}} V_2(\varrho ,\varphi ,\tau) \leq  (1+\varkappa ) m_+\delta_\varepsilon^2 <  (1-\varkappa ) m_- \tau_0^{-3/4}\varepsilon^2  \leq
\inf_{\{\varrho^2+ (\tau/\tau_0)^{-3/4} \varphi^2  =\varepsilon^2\}}V_2(\varrho ,\varphi ,\tau),
\end{gather*}
for all $1 \leq \tau/\tau_0\leq \varepsilon^{-8/3}d_0^{8/3}$, where $d_0=\min\{d_1,d_2\}$,
$\tau_0=\max\{\tau_1,\tau_2, 1\}$, $m_-=\min\{\lambda^{1/2},\omega^2_2/2\}$, and
 $m_+=\max\{\lambda^{1/2},\omega^2_2/2\}$.
This implies that any solution of system \eqref{hatham1} starting in $\mathcal B_{\delta_\varepsilon}$ at $\tau=\tau_0$ satisfies the inequalities:
\begin{gather*}
|\varrho (\tau)|< \varepsilon, \quad \tau^{-3/8}|\varphi (\tau)|< \varepsilon
\end{gather*}
as $\tau_0\leq \tau\leq \mathcal O(\varepsilon^{-8/3})$. Hence, the equilibrium $(0,0)$ of system \eqref{hatham1} is $\varrho$-stable at least on the asymptotically long time interval. Returning to the original variables $(\rho,\psi)$, we obtain the result of the theorem.
\end{proof}

Similarly, we have the following.
\begin{Th}
\label{c3s}
Let $\sigma$ be a root of equation \eqref{TEQ}. In {\bf Case III}, the solution $\rho_\ast(\tau)$, $\psi_\ast(\tau)$ with asymptotics \eqref{PAS3} is $\rho$-stable on a finite but asymptotically long time interval.
\end{Th}
\begin{proof}
The change of variables
\begin{gather}
\label{chv3}
R(\tau)= \varrho(\tau),\quad \Psi(\tau)=\tau^{-1/6}\varphi(\tau)
\end{gather}
transforms system \eqref{ham} into
\begin{gather}
    \label{hatham2}
    \frac{d\varrho}{d\tau}=-\partial_{\varphi } \hat H_3(\varrho ,\varphi ,\tau), \quad
    \frac{d\varphi }{d\tau}=\partial_{\varrho } \hat H_3(\varrho ,\varphi ,\tau)+\hat F_3(\varrho ,\varphi ,\tau),
\end{gather}
where
\begin{gather*}
\hat H_3(\varrho,\varphi,\tau)=   \tau^{1/6} H(\varrho,\tau^{-1/6}\varphi,\tau), \quad
\hat F_3(\varrho,\varphi,\tau)= \tau^{1/6} F (\varrho,\tau^{-1/6}\varphi,\tau) + \tau^{-1}\frac{\varphi}{6}.
\end{gather*}
By taking into account the asymptotics for the particular solution, we obtain
\begin{eqnarray*}
\hat H_3(\varrho ,\varphi ,\tau)&= &  \tau^{5/12} \lambda^{1/2} \varrho ^2 + \tau^{-1/4} \Big(\omega_3^2 \frac{\varphi ^2}{2}-\mathcal P'''(\sigma;\delta,\nu)\chi\frac{\varphi ^3}{6}+\mathcal P'''(\sigma;\delta,\nu) \frac{\varphi ^4}{24} \Big) + \mathcal O(\tau^{-1/3}),  \\
\hat F_3(\varrho ,\varphi ,\tau)&= & \tau^{-1} \Big(- \frac{\varphi }{12}+\mathcal P'''(\sigma;\delta,\nu)  \Big(\chi^2 \frac{\varphi ^2} {4 \lambda^{1/2}}-\chi \frac{\varphi ^3} {6 \lambda^{1/2}}+\frac{\varphi ^4}{24}\Big) \Big) +\mathcal O(  \tau^{-7/6})
\end{eqnarray*}
as $\tau\to\infty$. As above, we use the Hamiltonian $\hat H_3(\varrho,\varphi,\tau)$ as the basis for the Lyapunov function.
Consider the combination:
\begin{gather*}
V_3(\varrho ,\varphi ,\tau):= v_{3}^0(\varrho ,\varphi ,\tau)+\tau^{-17/12} v_{3}^1(\varrho ,\varphi ),
\end{gather*}
where
\begin{gather*}
v_{3}^0(\varrho ,\varphi ,\tau):=\tau^{-5/12}\hat H_3(\varrho ,\varphi ,\tau), \quad v_{3}^1(\varrho ,\varphi ):= -\frac{5}{24} \varrho \varphi.
\end{gather*}
It follows easily that for all $0<\varkappa <1$ there exist $d_1>0$ and $\tau_1>0$ such that
\begin{gather*}
 (1-\varkappa ) W_3^2(\varrho ,\varphi ,\tau) \leq V_3(\varrho ,\varphi ,\tau)\leq (1+\varkappa ) W_3^2(\varrho ,\varphi ,\tau)
\end{gather*}
for all $(\varrho,\varphi)\in\mathcal B_{d_1}$ and $\tau\geq \tau_1$, where
\begin{gather*}
W_3(\varrho ,\varphi ,\tau) := \sqrt{\lambda^{1/2}\varrho ^2+\tau^{-2/3}\omega_3^2\frac{\varphi ^2}{2}}.
\end{gather*}
Calculating the derivatives of the functions $v_{3}^0(\varrho ,\varphi ,\tau)$ and $v_{3}^1(\varrho ,\varphi )$ with respect to $\tau$ along the trajectories of   system \eqref{hatham2}, we obtain
\begin{eqnarray*}
    \frac{dv_{3}^0}{d\tau}\Big|_\eqref{hatham2} & = & \tau^{-5/12}   \Big(\frac{\partial \hat H_3}{\partial \tau}+ \frac{\partial \hat  H_3}{\partial \varphi }   \hat  F_3\Big) -\frac{5 }{12} \tau^{-17/12}\hat   H_3 \\
    &=& \tau^{-5/3}  \Big( - \omega_3^2 \frac{5\varphi ^2}{12}+\mathcal O(\hat  d^3)\Big)+ \mathcal O(\tau^{-11/6} \hat d^2), \\
    \tau^{-17/12}\frac{dv_{3}^1}{d\tau}\Big|_\eqref{hatham2} & = & -\tau^{-1} \lambda^{1/2}\frac{5 \varrho ^2}{12}  + \tau^{-5/3}\Big( \omega_3^2 \frac{5\varphi ^2}{24}+\mathcal O(\hat  d^3)\Big) +\mathcal O(\tau^{-11/6} \hat  d^2)
\end{eqnarray*}
as $\hat d\to 0$ and $\tau\to\infty$.
Combining the preceding estimates, we get
\begin{gather*}
\frac{dV_3}{d\tau}\Big|_\eqref{hatham2} = - \tau^{-1}  \frac{5}{12}\Big( \lambda^{1/2}    \varrho ^2 +\tau^{-2/3}\omega_3^2 \frac{\varphi ^2}{2} +\mathcal O(\hat d^3)\Big)+\mathcal O(\tau^{-11/6} \hat  d^2).
\end{gather*}
Hence, for all $0<\varkappa <1$ there exist $d_2>0$ and $\tau_2>0$ such that
\begin{gather}
\label{v3est}
\frac{dV_3}{d\tau}\Big|_\eqref{ham2} \leq -\tau^{-1 } l_\varkappa  \frac{5 }{12} V_3,
\end{gather}
for all $(\varrho,\varphi)\in\mathcal B_{d_2}$ and $\tau\geq \tau_2$. Therefore, as in the previous case, the fixed point $(0,0)$ of system \eqref{hatham2} is $\varrho $-stable on an asymptotically long time interval: for all $\varepsilon>0$ there exists $\delta_\varepsilon>0$ such that any solution of system \eqref{hatham2} starting from $\mathcal B_{\delta_\varepsilon}$ at $\tau=\tau_0$ satisfies the inequalities
\begin{gather*}
|\varrho (\tau)|<\varepsilon, \quad \tau^{-1/3}|\varphi (\tau)|<\varepsilon
\end{gather*}
as $1\leq \tau/\tau_0\leq d_0^{3}\varepsilon^{-3}$, where $d_0=\min\{d_1,d_2\}$ and $\tau_0=\max\{\tau_1,\tau_2,1\}$. Returning to the original variables, we see that the particular solution $\rho_\ast(\tau)$, $\psi_\ast(\tau)$ is $\rho$-stable on the asymptotically long time interval.

\end{proof}

The stability of the particular solutions $\rho_\ast(\tau)$, $\psi_\ast(\tau)$ ensures the existence of a family of autoresonant solutions with a similar behaviour (see Fig.~\ref{rhopsi}). Some rough estimates for these solutions follow directly from the properties of constructed Lyapunov functions.

\begin{figure}
\centering
\includegraphics[width=0.4\linewidth]{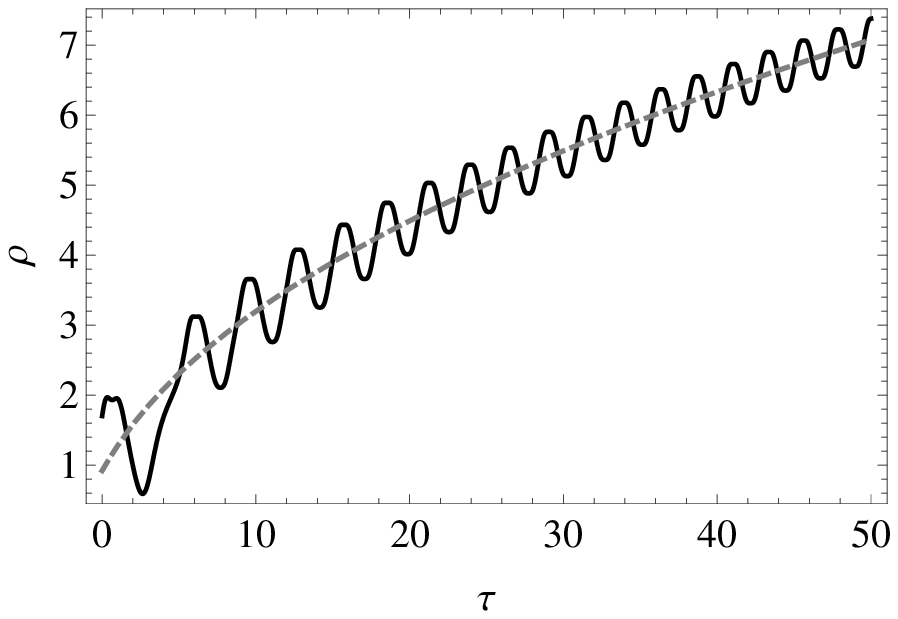}
\hspace{4ex}
\includegraphics[width=0.4\linewidth]{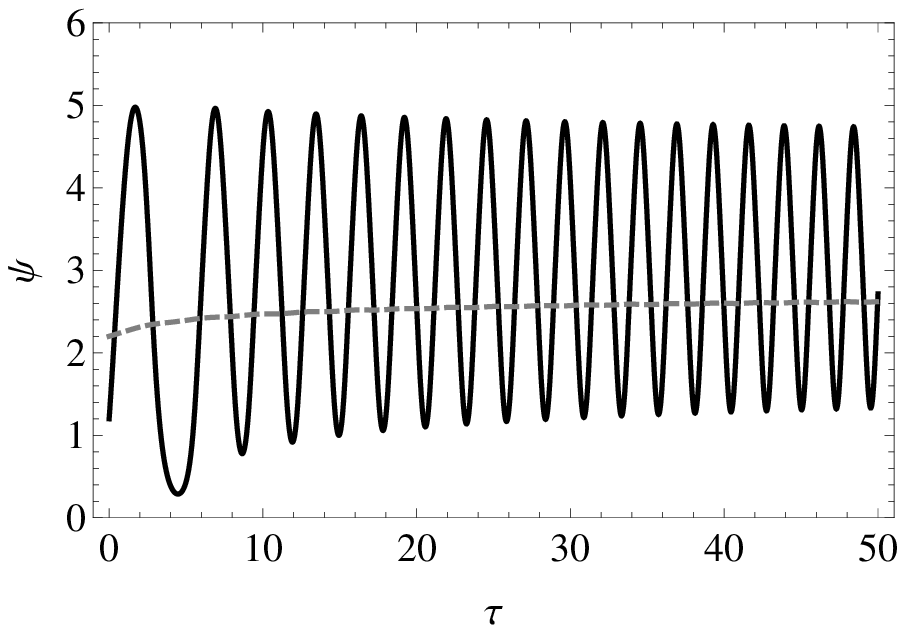}
\caption{\small The evolution of $\rho(\tau)$ and $\psi(\tau)$ for solutions of \eqref{MS} with $\lambda=1$, $\mu(\tau)\equiv \mu_0 (1+\tau)^{-1/2}$,  $\nu=0$, $\mu_0=-0.5$.} \label{rhopsi}
\end{figure}

We have the following.

\begin{Cor}
For all $\varkappa,\varepsilon\in (0,1)$ there exist $\Delta_0(\varkappa,\varepsilon)>0$ and $T_0(\varkappa)>0$ such that for all $(\rho^0,\psi^0)$: $(\rho^0-\rho_\ast(T_0))^2+(\psi^0-\psi_\ast(T_0))^2<\Delta_0^2$ the solution $\rho(\tau)$, $\psi(\tau)$ to system \eqref{MS} with initial data $\rho(T_0)=\rho^0$, $\psi(T_0)=\psi^0$ has the following estimates:
\begin{align}
\label{as1}& \rho=\sqrt{\lambda\tau}+\mathcal O(\tau^{-\frac{3+\varkappa}{8(1+\varkappa)}}), \ \  \psi=\sigma+\mathcal O(\tau^{-\frac{1-\varkappa}{8(1+\varkappa)}})  \ \ {\text as} \ \  \tau\to\infty \ \ \text{in {\bf Case I}};  \\
\label{as2} & \rho=\sqrt{\lambda\tau}+\mathcal O(\tau^{-\frac{7+\varkappa}{16(1+\varkappa)}}),  \ \  \psi=\sigma+\mathcal O(\tau^{-\frac{1-5\varkappa}{16(1+\varkappa)}})   \ \ {\text as} \ \  1\leq \frac{\tau}{T_0}\leq \mathcal O(\varepsilon^{-8/3}) \ \  \text{in {\bf Case II}}; \\
\label{as3} & \rho=\sqrt{\lambda\tau}+\mathcal O(\tau^{-\frac{11+\varkappa}{24(1+\varkappa)}}),  \ \  \psi=\sigma+\mathcal O(\tau^{-\frac{1-9\varkappa}{24(1+\varkappa)}})  \ \ {\text as} \ \  1\leq \frac{\tau}{T_0}\leq \mathcal O(\varepsilon^{-3}) \ \  \text{in {\bf Case III}}.
\end{align}
\end{Cor}
\begin{proof}
Let us fix $\varkappa,\varepsilon\in (0,1)$.
Consider {\bf Case I}. Let $R(\tau)$, $\Psi(\tau)$  be a solution to system \eqref{ham} starting from the ball $\mathcal B_{\Delta_0}$ at $\tau=T_0$, where $\Delta_0=\delta_\varepsilon$ and $T_0=\tau_0$  (see Theorem~\ref{cgs}). Then it follows from \eqref{V1est} that the function $v_1(\tau):=V_1(R(\tau),\Psi(\tau),\tau)$ satisfies the inequality:
\begin{gather*}
        \frac{d v_1}{d\tau}\leq -\tau^{-1}l_\varkappa \frac{v_1}{4}
\end{gather*}
as $\tau\geq T_0$, where $l_\varkappa=(1-\varkappa)/(1+\varkappa)$.
Integrating the last expression with respect to $\tau$, we obtain $0\leq v_1(\tau)\leq v_1(T_0) (\tau/T_0)^{-l_\varkappa /4}$, where $0\leq v_1(T_0) \leq C_0\Delta_0^2$, $C_0={\hbox{\rm const}}$. Thus we have $w_1\big(R(\tau),\Psi(\tau)\big)=\mathcal O(\tau^{-l_\varkappa /8})$ as $\tau\to\infty$.  Returning to the original variables $(\rho,\psi)$, we derive \eqref{as1}.

Consider {\bf Case II}. Let $\varrho(\tau)$, $\varphi(\tau)$ be a solution to system \eqref{hatham1} starting from $\mathcal B_{\Delta_0}$ at $\tau=T_0$, where $\Delta_0=\delta_\varepsilon$, $T_0=\tau_0$ (see Theorem~\ref{c2s}). From \eqref{v2est} it follows that the derivative of the function $v_2(\tau):=V_2\big(\varrho (\tau),\varphi (\tau),\tau\big)$ satisfies the following inequality:
\begin{gather*}
     \frac{d v_2}{d\tau}\leq -\tau^{-1}l_\varkappa  \frac{3 v_2}{8}
\end{gather*}
as $1\leq \tau/T_0 \leq \mathcal O(\varepsilon^{-8/3})$.
By integrating the last estimate with respect to $\tau$, we get $0\leq v_2(\tau)\leq v_2(T_0) (\tau/T_0)^{-3 l_\varkappa /8}$,  where $0\leq v_2(T_0) \leq C_0\Delta_0^2$, $C_0={\hbox{\rm const}}$. Hence, $W_2(\varrho (\tau),\varphi (\tau),\tau)=\mathcal O(\tau^{-3 l_\varkappa /16})$ as $1\leq \tau/T_0\leq \mathcal O(\varepsilon^{-8/3})$.  The change of variables \eqref{chv2} transforms the last estimate into \eqref{as2}.

Finally, consider {\bf Case III}. Let $\varrho (\tau)$, $\varphi (\tau)$ be a solution to system \eqref{hatham2} with initial data from the domain $\mathcal B_{\Delta_0}$ at $\tau=T_0$, where $\Delta_0=\delta_\varepsilon$ and  $T_0=\tau_0$ (see Theorem~\ref{c3s}). From \eqref{v3est} it follows that the function $v_3(\tau):=V_3\big(\varrho (\tau),\varphi (\tau),\tau\big)$ satisfies the inequality:
\begin{gather*}
     \frac{d v_3}{d\tau}\leq -\tau^{-1} l_\varkappa  \frac{5 v_3}{12}
\end{gather*}
as $1\leq \tau/T_0\leq \mathcal O(\varepsilon^{-3})$.
As in the previous case, by integrating the last inequality, we obtain
$0\leq v_3(\tau)\leq v_3(T_0) (\tau/T_0)^{-5 l_\varkappa /12}$,  where $0\leq v_3(T_0) \leq C_0\Delta_0^2$, $C_0={\hbox{\rm const}}$. Hence, $W_3(\varrho (\tau),\varphi (\tau),\tau)=\mathcal O(\tau^{-5 l_\varkappa /24})$ as $1\leq \tau/T_0\leq \mathcal O(\varepsilon^{-3})$.  Combining this with \eqref{chv3}, we get \eqref{as3}.
\end{proof}

\section{Asymptotic analysis}

In this section, the asymptotics for general autoresonant solutions starting from a neighborhood of stable solutions are specified by a modified averaging method~\cite{BM61,ANJJAMM84,BDP01,AKN06} with using the constructed Lyapunov functions.

Asymptotics are most simply constructed in {\bf Case I}, when $\sigma$ is a simple root to equation \eqref{TEQ}. We have the following (see~\cite{OSSIAM18}).
\begin{Th}
\label{Thas0}
Let $\sigma$ be a simple root to equation \eqref{TEQ}. Then, in {\bf Case I}, system \eqref{MS} has two-parameter family of autoresonant solutions $\rho_{I}(\tau;\alpha_0,c)$, $\psi_{I}(\tau;\alpha_0,c)$ with the asymptotics
\begin{eqnarray*}
\rho_{I}(\tau;\alpha_0,c)&=&\sqrt{\lambda \tau}+ c \lambda^{-1/4} \tau^{-3/8} \cos \tilde A(\tau;\alpha_0)+\mathcal O(\tau^{-1/2}), \\
\psi_{I}(\tau;\alpha_0,c)&=&\sigma+  c \sqrt 2 \omega_1^{-1}\tau^{-1/8}\sin \tilde  A(\tau;\alpha_0) + \mathcal O(\tau^{-1/4})
\end{eqnarray*}
as $\tau\to\infty$,
where the function $\tilde A(\tau;\alpha_0)$ has the following form{\rm :}
\begin{gather*}
    \tilde A(\tau;\alpha_0)=\alpha_0+(4\lambda)^{1/4}\omega_1 \frac{ 4 }{5}\tau^{5/4}+\sum_{k=1}^4 \alpha_{k}(c)\tau^{k/4}+\alpha_{-1}(c)\log \tau, \\
   \omega_1=\sqrt{\mathcal P'(\sigma;\delta,\nu)}, \quad \alpha_4(c)=\frac{c^2 \sqrt\theta}{48 \omega_1^4}\Big(12  \omega_1^2 \mathcal P'''(\sigma;\delta,\nu)\mathcal -5(\mathcal P''(\sigma;\delta,\nu))^2\Big), \quad \alpha_k(c)={\hbox{\rm const}}.
\end{gather*}
\end{Th}

In other cases, the stability of the particular solutions $\rho_\ast(\tau)$, $\psi_\ast(\tau)$ has not been justified for all $\tau\geq 0$. Therefore, the asymptotics for general autoresonant solutions are constructed only on the asymptotically long time intervals.

\begin{Th}
\label{Thas1}
Let $\sigma$ be a root of multiplicity 2 to equation \eqref{TEQ}. Then, in {\bf Case II}, there exists $T_0>0$ such that a two-parameter family of autoresonant solutions $\rho_{II}(\tau;\alpha_0,\varepsilon)$, $\psi_{II}(\tau;\alpha_0,\varepsilon)$ to system \eqref{MS}, starting at $\tau=T_0$ from $\varepsilon$-neighbourhood of the solution $\rho_\ast(\tau)$, $\psi_\ast(\tau)$, has the following asymptotics{\rm :}
\begin{eqnarray*}
\rho_{II}(\tau;\alpha_0,\varepsilon)&=&\rho_\ast(\tau)+ \varepsilon\lambda^{-1/4}  \tau^{-{7}/{16}}  {\tilde l_0(\tau) }\cos \tilde A(\tau;\alpha_0)+O(\varepsilon^{7/6}), \\
\psi_{II}(\tau;\alpha_0,\varepsilon)&=&\psi_\ast(\tau)+\varepsilon\sqrt 2 \omega_2^{-1}   \tau^{-{1}/{16}}  {\tilde l_0(\tau)} \sin \tilde A(\tau;\alpha_0)+O(\varepsilon^{7/6})
\end{eqnarray*}
as $\varepsilon\to 0$ uniformly for $\tau/T_0\in (1,\mathcal O(\varepsilon^{-8/3}))$, where $\tilde  A(\tau;\alpha_0)$ has the following form{\rm :}
\begin{gather*}
       \tilde A(\tau;\alpha_0)=\alpha_0+ (4\lambda)^{1/4} \omega_2  \frac{8}{9}\tau^{9/8} +\sum_{k=1}^3 \varepsilon^{2k} \tilde\alpha_k(\tau),
\end{gather*}
with $\omega_2=\sqrt{-\phi \mathcal P''(\sigma;\delta,\nu)}$, $\tilde\alpha_k(\tau)=\mathcal O(\tau^{(9+3k)/8})$ and $\tilde l_{0}(\tau)=1+\mathcal O(\tau^{-1/8})$ as $\tau\to\infty$.
\end{Th}
\begin{proof}
Let $\rho_\ast(\tau)$, $\psi_\ast(\tau)$ be a particular solution to system \eqref{MS} with asymptotics \eqref{PAS2}, where $\psi_1=-\phi$. We apply the change of variables
\begin{gather*}
\rho(\tau)=\rho_\ast(\tau)+ \tau^{-5/8} r(s),\quad \psi(\tau)=\psi_\ast(\tau)+\tau^{-1/4} p(s), \quad s=\frac{8}{9} \tau^{9/8}
\end{gather*}
in system \eqref{MS} and study the solutions of transformed system \eqref{ham1} in a neighborhood of the unstable equilibrium $(0,0)$.
Consider the Hamiltonian system:
\begin{gather*}
    \frac{d r}{d \alpha}=-\partial_p  h_2^0(r,p), \quad \frac{d p}{d \alpha}=\partial_r h_2^0(r,p),
\end{gather*}
where
\begin{gather*}
h_2^0(r,p)=\lim\limits_{s\to\infty}\mathcal H_2(r,p,s)= \lambda^{1/2} r^2 +  \omega_2^2 \frac{p^2}{2}+\mathcal P''(\sigma;\delta,\nu)\frac{p^3}{6}.
 \end{gather*}
It is follows from the definition of the function $h_2^0(r,p)$ that the level lines $\{(r,p)\in\mathbb R^2:   h_2^0(r,p)\equiv I\}$ define a family of closed curves on the phase space $(r,p)$ parameterized by the parameter $I\in (0, I_\ast)$, $I_\ast=2(\omega_2 \phi)^2/3$. It can easily be checked that to each closed curve there corresponds a periodic solution $\xi(\alpha,I)$, $\eta(\alpha,I)$ of period $T(I)=2\pi (\omega(I))^{-1}$,  where
\begin{gather*}
\omega(I)=(4\lambda)^{1/4}\omega_2 -\frac{5 I}{48 (\omega_2 \phi)^2} +\mathcal O(I^{2}), \quad I\to 0.
\end{gather*}
These solutions are used in the definition of the functions
\begin{gather*}
      \tilde r(\alpha,I)=\xi\Big(\frac{\alpha}{\omega},I\Big), \quad
      \tilde p(\alpha,I)=\eta\Big(\frac{\alpha}{\omega},I\Big),
\end{gather*}
that are $2\pi$-periodic with respect to $\alpha$. Consider the change of variables
\begin{gather}
\label{chrp2}
    r(s)=\tilde r\big(\alpha(s),I(s)\big), \quad
    p(s)=\tilde p\big(\alpha(s),I(s)\big)
\end{gather}
in system \eqref{ham1}. Since
\begin{gather*}
\omega(I)\frac{\partial \tilde r}{\partial \alpha} = -\partial_p h_2^0 (\tilde r,\tilde p), \quad \omega(I)\frac{\partial \tilde p}{\partial \alpha} = \partial_r h_2^0 (\tilde r,\tilde p), \\
   \frac{d}{dI} h_2^0(\tilde r,\tilde p)=\omega(I)\Big(\partial_I \tilde r \partial_\alpha \tilde p - \partial_\alpha \tilde r \partial_I \tilde p  \Big) \equiv 1,
\end{gather*}
the transformation of variables \eqref{chrp2} is reversible while $\omega(I)\neq 0$.
It can easily be checked that the system in the action-angle variables $(I,\alpha)$ has the following form:
\begin{gather}
\label{sysIa2}
  \frac{dI}{ds} = \mathcal  F^I(\alpha,I,s), \quad \frac{d\alpha}{ds} =   \mathcal  F^\alpha(\alpha,I,s),
\end{gather}
where
\begin{eqnarray*}
   \mathcal   F^I(\alpha,I,s) &  := & -   \omega(I) \Big(\partial_\alpha \widetilde {\mathcal H}(\alpha,I,s) +\partial_\alpha \tilde r (\alpha,I) \widetilde {\mathcal F}(\alpha,I,s)\Big), \\
 \mathcal   F^\alpha(\alpha,I,s) & :=&  \omega(I)  \Big( \partial_I \widetilde {\mathcal H}(\alpha,I,s) +\partial_I \tilde r (\alpha,I) \widetilde {\mathcal F}(\alpha,I,s)\Big), \\
\widetilde {\mathcal F}(\alpha,I,s)&:= &\mathcal F(\tilde r(\alpha,I),\tilde p(\alpha,I),s) \\
\widetilde {\mathcal H}(\alpha,I,s)&:= &\mathcal H(\tilde r(\alpha,I),\tilde p(\alpha,I),s).
\end{eqnarray*}
Note that $\mathcal F^I(\alpha,I,s)$ and $\mathcal F^\alpha(\alpha,I,s)$ are $2\pi$-periodic functions with respect to $\alpha$ such that
\begin{gather*}
\mathcal F^I(\alpha,I,s)=\mathcal O(s^{-2/9}), \quad \mathcal F^\alpha(\alpha,I,s)=\omega(I)+\mathcal O(s^{-2/9})
\end{gather*}
as $s\to\infty$ and for all $\alpha\in\mathbb R$ and $I\in [0,I_\ast]$.
To simplify system \eqref{sysIa2}, we introduce a new dependent variable $\mathcal L(s)$ associated with the Lyapunov function \eqref{UF1} such that
\begin{gather}
\label{exchLF2}
\mathcal L(s):=\widetilde U_2(\alpha(s),I(s),s),
\end{gather}
where $\widetilde U_2(\alpha,I,s):= U_2(\tilde r(\alpha,I),\tilde p(\alpha,I),s)$ is $2\pi$-periodic function in $\alpha$.
It can easily be checked that $\widetilde U_2(\alpha,I,s)=I+\mathcal O( s^{-2/9})$  as $s\to \infty$ for all $\alpha\in \mathbb R$ and $I\in [0,I_\ast]$. Hence, the transformation $(I,\alpha)\mapsto (\mathcal L,\alpha)$ is reversible for all $L\in [0,L_\ast]$, $L_\ast={\hbox{\rm const}}>0$ and $\alpha\in \mathbb R$. The transformed system is given by
\begin{gather}
\label{sysLa2}
    \frac{d\mathcal L}{ds}=  \mathcal G(\alpha,\mathcal L,s), \quad \frac{d\alpha}{d s}=\mathcal Q(\alpha,\mathcal L,s),
\end{gather}
where
\begin{gather}
\begin{split}
\label{GQ}
    \mathcal G(\alpha,\widetilde U_2(\alpha,I,s),s)& \equiv \partial_\alpha \widetilde U_2 (\alpha,I,s)\frac{d\alpha}{ds} + \partial_I \widetilde U_2(\alpha,I,s) \frac{dI}{ds}+\partial_s \widetilde U_2 (\alpha,I,s) \\
        & = \partial_r U_2(\tilde r, \tilde p,s) \frac{dr}{ds}+  \partial_p U_2(\tilde r, \tilde p,s) \frac{dp}{ds}+ \partial_s U_2(\tilde r, \tilde p,s) \\
        & = \frac{d}{ds} \Big|_{\eqref{ham1}}U_2(\tilde r(\alpha,I), \tilde p(\alpha,I),s) \\
     \mathcal Q(\alpha,\widetilde U_2(\alpha,I,s),s)&\equiv   \mathcal F^\alpha(\alpha,I,s).
    \end{split}
\end{gather}
It is not difficult to deduce from \eqref{GQ} and \eqref{dU2} the asymptotics of the functions $\mathcal G$ and $\mathcal Q$ at infinty:
\begin{gather*}
      \mathcal G(\alpha,\mathcal L,s)=s^{-1}\sum_{k=0}^\infty  g_k(\alpha,\mathcal L) s^{-k/9}, \quad    \mathcal Q(\alpha,\mathcal L,s)=\omega(\mathcal L)+\sum_{k=2  }^\infty q_k(\alpha,\mathcal L) s^{-k/9},
\end{gather*}
where $g_k(\alpha,\mathcal L)$ and $q_k(\alpha,\mathcal L) $ are $2\pi$-periodic functions with respect to $\alpha$, and $g_0(\alpha,\mathcal L)= \mathcal L/3+\mathcal O(\mathcal L^{3/2})$ as $\mathcal L\to 0$ for all $\alpha\in\mathbb R$.
For the convenience, we rewrite system \eqref{sysLa2} in a near-Hamiltonian form:
\begin{gather}
\label{sysLa2Ham}
    \frac{d\mathcal L}{ds}= -\partial_\alpha \mathcal M + \mathcal J, \quad \frac{d\alpha}{ds}=\partial_\mathcal L \mathcal M,
\end{gather}
where
\begin{gather*}
   \mathcal M(\alpha,\mathcal L,s)=\int\limits_0^{\mathcal L}  \mathcal Q(\alpha,l,s)\, d l, \quad  \mathcal J(\alpha,\mathcal L,s)=\mathcal G(\alpha,\mathcal L,s)+\partial_\alpha \mathcal M(\alpha,\mathcal L,s).
\end{gather*}
The asymptotic solution to the first equation in \eqref{sysLa2Ham} is sought in the form:
\begin{gather*}
    \mathcal L(s)=L (s)+\ell (\alpha, L(s),s),
\end{gather*}
where
$L(s)$ is determined from the averaged equation
\begin{gather}
\label{averA1}
    \frac{dL}{ds}= \Big\langle \mathcal J\big(\alpha,L+\ell(\alpha,L,s),s\big)\Big\rangle_\alpha.
\end{gather}
Then $\ell(\alpha,L,s)$ satisfies the equation:
\begin{gather}
\label{Meq1}
    \frac{d}{d\alpha} \mathcal M\big(\alpha,L+\ell(\alpha,L,s),s\big)=\mathcal F^{\mathcal M}\big(\alpha,L,\ell(\alpha,L,s),s\big)
\end{gather}
with
\begin{eqnarray*}
    \mathcal F^{\mathcal M}:=  \mathcal J(\alpha,L+\ell(\alpha,L,s),s) -\big(1+\partial_{L} \ell(\alpha,L,s)\big) \Big\langle \mathcal J(\alpha,L+\ell(\alpha,L,s),s\big)\Big\rangle_\alpha-\partial_s \ell(\alpha,L,s).
\end{eqnarray*}
In addition, it is assumed that $\ell(\alpha,L,s)$ is a $2\pi$-periodic function with zero average (with respect to $\alpha$):
\begin{gather*}
    \big\langle \ell(\alpha,L,s)\big\rangle_\alpha:=\frac{1}{2\pi}\int\limits_0^{2\pi} \ell (a,L,s)\, da\equiv 0,
\end{gather*}
Equation \eqref{Meq1} can be integrated with respect to $\alpha$ by choosing the constant of integration in such a way that the result has a zero average:
\begin{gather}
\label{int1}
\begin{split}
  \mathcal M\big(\alpha,L+\ell(\alpha,L,s),s\big)- \Big \langle  \mathcal M\big(\alpha,L+\ell(\alpha,L,s),s\big)\Big\rangle_\alpha \\
  = \int \mathcal F^{\mathcal M}\big(\alpha,L,\ell(\alpha,L,s),s\big) \,d\alpha - \Big \langle  \int\limits \mathcal F^{\mathcal M}\big(\alpha,L,\ell(\alpha,L,s),s\big) \,d\alpha\Big\rangle_\alpha
\end{split}
\end{gather}
The asymptotic solution to equation \eqref{int1} is constructed in the form:
\begin{gather}
\label{serell1}
        \ell (\alpha,L,s)=s^{-1}\sum_{k=0}^\infty \ell_k(\alpha,L)s^{-k/9} .
\end{gather}
Substituting this series into equation \eqref{int1} and equating the terms of the same power of $s$, we obtain the following chain of equations:
\begin{gather*}
    \omega (L)   \ell_k  = \Lambda_k (\alpha,L),\quad k\geq 0,
\end{gather*}
 where each function $\Lambda_k(\alpha,L)$, $k\geq 1$ is expressed through $\ell_0$, $\dots$, $\ell_{k-1}$ such that $\langle \Lambda_k(\alpha,L)\rangle_\alpha =0$. For example,
\begin{eqnarray*}
    \Lambda_0 (\alpha,L)
        & = &
            \int g_0(\alpha,L)-\langle g_0(\alpha,L) \rangle_\alpha\, d\alpha-\Big\langle\int g_0(\alpha,L)-\langle g_0(\alpha,L) \rangle_\alpha \,d\alpha\Big\rangle_\alpha,\\
    \Lambda_1 (\alpha,L)
        & = &
            \int g_1(\alpha,L)-\langle g_1(\alpha,L) \rangle_\alpha\, d\alpha-\Big\langle\int g_1(\alpha,L)-\langle g_1(\alpha,L) \rangle_\alpha \,d\alpha\Big\rangle_\alpha,\\
    \Lambda_2 (\alpha,L)
        & = &
            \int g_2(\alpha,L)-\langle g_2(\alpha,L) \rangle_\alpha\, d\alpha-\Big\langle\int g_2(\alpha,L)-\langle g_2(\alpha,L) \rangle_\alpha \,d\alpha\Big\rangle_\alpha\\
        &   &
            +\int \ell_0 \partial_\alpha q_2 - \langle \ell_0  \partial_\alpha q_2 \rangle \, d\alpha - \Big\langle\int \ell_0 \partial_\alpha q_2 - \langle \ell_0  \partial_\alpha q_2 \rangle \, d\alpha \Big\rangle_\alpha \\
        &   &
            + \big(\partial_L \ell_0 - \langle \partial_L \ell_0 \rangle_\alpha\big)\Big\langle \int\limits_0^L \partial_\alpha q_2(\alpha,l)\, dl\Big\rangle_\alpha -\ell_0 q_2 +\langle \ell_0 q_2\rangle_\alpha.
\end{eqnarray*}
Thus, all coefficients $\ell_k$ are uniquely determined in the class of $2\pi$-periodic functions with zero average $\langle \ell_k(\alpha,L)\rangle_\alpha=0$.

In the same way, the solution to the second equation in \eqref{sysLa2Ham} is sought in the form:
\begin{gather*}
    \alpha(s)= A(s)+\theta\big(A(s),L(s),s\big),
\end{gather*}
where $A(s)$ is determined from the averaged equation
\begin{gather}
\label{avs1}
    \frac{d A}{ds}=\big\langle \mathcal K(A,\theta,L,s)\big\rangle_{ A},
\end{gather}
$\mathcal K:= \partial_{\mathcal L} \mathcal M\big(A+\theta(A,L,s),L+\ell(A+\theta(A,L,s),L,s),s\big)$.
The function $\theta(A,L,s)$ satisfies the following equation:
\begin{gather}
\label{psieq1}
 \big\langle \mathcal K(A,\theta,L,s) \big\rangle_{A}\,\frac{\partial \theta }{\partial A}=\mathcal N(A,\theta,L,s),
\end{gather}
where
$
\mathcal N:= \mathcal K(A,\theta,L,s)-\langle \mathcal K(A,\theta,L,s)\rangle_A-\partial_s\theta(A,L,s)-\partial_L\theta(A,L,s) \big\langle \mathcal J(\alpha,L+\ell(\alpha,L,s),s\big)\big\rangle_\alpha
$.
The asymptotic solution to \eqref{psieq1} is constructed in the form:
\begin{gather}
\label{sertheta1}
    \theta(A,L,s)=\sum_{k=2}^{\infty} \theta_k(A,L)s^{-k/9}
\end{gather}
with the additional condition: $\langle \theta_k(A,L)\rangle_{A}=0$.
The substitution the series into equation \eqref{psieq1} and the grouping the expressions of the same power of $s$ give the following chain of differential equations:
\begin{gather*}
\omega(L)\partial_{A} \theta_k=\Theta_k(A,L) - \langle \Theta_k(A,L)\rangle_{A}, \quad k\geq 2,
\end{gather*}
where each function $\Theta_k(A,L)$ for $k\geq 4$ is expressed through $\theta_2$, $\dots$, $\theta_{k-2}$. For example, \begin{eqnarray*}
    \Theta_2 &=&q_2(A,L), \\
    \Theta_3 &=& q_3(A,L), \\
    \Theta_4 & = & q_4(A,L) + \theta_2  \partial_A q_2(A,L)-\partial_A \theta_2 \langle q_2(A,L)\rangle_A -\partial_L \theta_2 \Big\langle \int\limits_0^L \partial_\alpha q_2 (\alpha,l)\, dl\Big\rangle_\alpha, \\
    \Theta_5 & = & q_5(A,L)+\theta_2 \partial_A q_3(A,L) +\theta_3 \partial_A q_2(A,L) -\partial_A \theta_2\langle q_3(A,L)\rangle_A-\partial_A \theta_3\langle q_2(A,L)\rangle_A\\
            & & - \partial_L \theta_3 \Big\langle \int\limits_0^L \partial_\alpha q_2 (\alpha,l)\, dl\Big\rangle_\alpha - \partial_L \theta_2 \Big\langle \int\limits_0^L \partial_\alpha q_3 (\alpha,l)\, dl\Big\rangle_\alpha
\end{eqnarray*}
It follows easily that all coefficients $\theta_k$ are uniquely determined in the class of $2\pi$-periodic functions with $\langle \theta_k(A,L)\rangle_{A}=0$.

In the last step, we integrate the averaged equations \eqref{averA1} and \eqref{avs1}.
First note that substituting series \eqref{serell1} and \eqref{sertheta1} for $\ell$ and $\theta$ in right-hand sides of the averaged equations, we get
\begin{eqnarray*}
    \widehat J(L,s)&:=& \Big\langle \mathcal J\big(\alpha,L+\ell(\alpha,L,s),s\big)\Big\rangle_\alpha
        =
            s^{-1}\sum_{i=0}^\infty \widehat J_i(L)s^{-i/9},\\
    \widehat K(L,s)& := & \big\langle \mathcal K(A,\theta,L,s)\big\rangle_{ A}
        = \omega(L)+
            \sum_{i=2}^\infty \widehat K_i(L)s^{-i/9}
\end{eqnarray*}
as $s\to\infty$, where $\widehat J_0(L)=\langle g_0(\alpha,L) \rangle_\alpha$, $\widehat K_2(L)=\langle q_2(A,L)\rangle_A$, etc.
Consider the following system of two differential equations:
\begin{gather}
\label{sysLAas2}
\frac{dL}{ds}= \widehat J(L,s), \quad \frac{dA}{ds}= \widehat K(L,s), \quad s\geq s_0,
\end{gather}
$s_0={\hbox{\rm const}}>0$. Since $\widehat J_0(L)=L/3+\mathcal O(L^2)$ and $\widehat J_k(L)=\mathcal O(L)$, $k\geq 1$ as $L\to 0$, then every solution with initial data close to zero
 escapes from the domain $(0,L_\ast)$ at time $S_\ast=S_\ast(s_0,L(s_0))$. Note that the asymptotic approximation basing on the Lyapunov function is not valid as $s\geq S_\ast$.
Consider a one parametric family of solutions to the first equation in \eqref{sysLAas2} with initial data $L(s_0)=\mathcal O(\varepsilon^2)$, where $\varepsilon$ is a small positive parameter. The asymptotic solution is sought in the form:
\begin{gather}
\label{Las1}
L(s)=\varepsilon^2 \sum_{k=0}^\infty\varepsilon^{2k} L_{k}(s).
\end{gather}
Substituting this into the equation and equating coefficients of powers
of $\varepsilon$, we get $L_k(s)= s^{(k+1)/3} l_k(s)$, where $l_{k}(s)=1+\mathcal O(s^{-1/9})$ as $s\to\infty$.
Hence,
\begin{gather*}
L(s)=\varepsilon^2  s^{1/3}l_0(s)(1+\mathcal O(\varepsilon)),  \quad \varepsilon\to 0
\end{gather*}
uniformly for $1\leq s/s_0\leq \mathcal O(\varepsilon^{-3})$.
$A(s)$ is found by integrating the second equation in \eqref{sysLAas2} with respect to $s$:
\begin{gather}
\label{Aform1}
     A(s)=\alpha_0+ \int\limits_{s_0}^s \widehat K(L(\varsigma),\varsigma)\,  d\varsigma,
\end{gather}
where $\alpha_0$ is the arbitrary parameter: $A(s_0)=\alpha_0$. Substituting \eqref{Las1} into \eqref{Aform1}, we obtain
\begin{gather*}
A(s)=A_0(s;\varepsilon)+\mathcal O(\varepsilon^{2/3}),  \quad \varepsilon\to 0,
\end{gather*}
uniformly for $1\leq s/s_0\leq \mathcal O(\varepsilon^{-3})$, where
\begin{eqnarray*}
A_0(s;\varepsilon)&:=&\alpha_0+(4\lambda)^{1/4} \omega_2 s+\varepsilon^2 \alpha_1(s)+\varepsilon^4 \alpha_2(s)+\varepsilon^6\alpha_3(s),\\
\alpha_1(s)&:=&\int\limits_{s_0}^s -\frac{5 }{48(\omega_2\phi)^2} L_0(\varsigma) +\sum_{k=2}^7  \widehat g_k '(0) \int\limits_{s_0}^s \varsigma^{-k/9} L_0(\varsigma)\, d\varsigma,\\
\alpha_2(s)&:=&\int\limits_{s_0}^s  \omega''(0) \frac{L_0^2(\varsigma)}{2}+\omega'(0) L_1(\varsigma)+\sum_{k=2}^4 \varsigma^{-k/9} \Big( \widehat g_k '(0) L_1(\varsigma)+   \widehat g_k ''(0) \frac{L_0^2(\varsigma)}{2}\Big) \, d\varsigma,\\
\alpha_3(s)&:=&\int\limits_{s_0}^s  \omega'''(0) \frac{L_0^3(\varsigma)}{6}+\omega''(0) L_0(\varsigma) L_1(\varsigma)+\omega'(0)L_2(\varsigma) \, d\varsigma,
\end{eqnarray*}
$\alpha_k(s)=\mathcal O(s^{(k+3)/3})$ as $s\to\infty$.
Combining this with \eqref{serell1}, \eqref{sertheta1}, \eqref{Las1}, \eqref{chrp2} and  \eqref{exchLF2}, we get the following asymptotic approximation of solutions to system \eqref{ham1}:
\begin{eqnarray*}
   r(s)&=& \varepsilon s^{1/6} \lambda^{-1/4} \sqrt{l_0(s)} \cdot \cos A_0(s;\varepsilon)+\mathcal O(\varepsilon^{7/6}), \\
   p(s)&=& \varepsilon  s^{1/6} \sqrt 2   \omega_2^{-1} \sqrt{l_0(s)} \cdot \sin A_0(s;\varepsilon)+\mathcal O(\varepsilon^{7/6}),
\end{eqnarray*}
as $\varepsilon\to 0$ uniformly for $s/s_0\in (1,\mathcal O(\varepsilon^{-3}))$.
Returning to the original variables we obtain the result of the theorem.
\end{proof}

Similarly, we have the following.

\begin{Th}
\label{Thas2}
Let $\sigma$ be a root of multiplicity 3 to equation \eqref{TEQ}. Then, in {\bf Case III}, there exists $T_0>0$ such that a two-parameter family of autoresonant solutions $\rho_{III}(\tau;\alpha_0,\varepsilon)$, $\psi_{III}(\tau;\alpha_0,\varepsilon)$ to system \eqref{MS}, starting at $\tau=T_0$ from $\varepsilon$-neighbourhood of the solution $\rho_\ast(\tau)$, $\psi_\ast(\tau)$, has the following asymptotics{\rm :}
\begin{eqnarray*}
\rho_{III}(\tau;\alpha_0,\varepsilon)&=&\rho_\ast(\tau)+ \varepsilon\lambda^{-1/4}  \tau^{-{11}/{24}} \tilde l_0(\tau) \cos \tilde A(\tau;\alpha_0)+O(\varepsilon^{9/8}), \\
\psi_{III}(\tau;\alpha_0,\varepsilon)&=&\psi_\ast(\tau)+\varepsilon\sqrt 2 \omega_3^{-1}   \tau^{-{1}/{24}} \tilde l_0(\tau) \sin \tilde A(\tau;\alpha_0)+O(\varepsilon^{9/8})
\end{eqnarray*}
as $\varepsilon\to 0$ uniformly for $\tau/T_0\in (1,\mathcal O(\varepsilon^{-3}))$, where $\tilde A(\tau;\alpha_0)$ has the following form{\rm :}
\begin{gather*}
        \tilde A(\tau;\alpha)=\alpha_0 - (2 \lambda)^{1/4}\frac{12}{13}\tau^{13/12} + \varepsilon^2 \tilde \alpha_1(\tau)+ \varepsilon^4 \tilde\alpha_2(\tau),
\end{gather*}
with $\omega_3=\sqrt{\chi^2 \mathcal P'''(\sigma;\delta,\nu)/2}$, $\tilde\alpha_k(\tau)=\mathcal O(\tau^{(3k+13)/12})$, $\tilde l_0(\tau)=1+\mathcal O(\tau^{-1/12})$ as $\tau\to\infty$.
\end{Th}
\begin{proof}
The proof is similar to the proof of Theorem \ref{Thas1} with using the Lyapunov function $U_3(r,p,s)$ instead of $U_2(r,p,s)$.
\end{proof}

\section{Conclusion}

In summary, the model of the autoresonant capture in nonlinear systems with the combined external and parametric excitation in the vicinity of the bifurcation points has been investigated. The suggested approach relies on the stability analysis of the particular solutions with power-law asymptotics at infinity. It has been shown that outside the bifurcation points there are several autoresonant modes with different phase shifts $\sigma_i$ associated with simple roots to equation \eqref{TEQ}. Depending on the sign of the value $\mathcal P''(\sigma_i;\delta,\nu)$, where $\delta=\mu_0\sqrt \lambda$, some of these modes are stable (see the shaded areas in Fig.~\ref{S0P}). To each stable mode there corresponds the two-parameter family of the autoresonant solutions with the asymptotics detailed in Theorem~\ref{Thas0}. Some of the autoresonant modes coalesce, when the parameters $(\delta,\nu)$ passes through the bifurcation curves $\gamma_\pm$ from $\Omega_+$ to $\Omega_-$. Assume that equation \eqref{TEQ} has only two different roots at the bifurcation point: $\sigma_0$ is a root of multiplicity 3 and $\sigma_1$ is a simple root. Then there are two autoresonant modes with different phase shifts. The stability of the mode corresponding to a multiple root depends on the sign of the value $\mathcal P'''(\sigma_0;\delta,\nu)$ (see Theorem \ref{c3s} and the shaded area in Fig.~\ref{M23}, a). In this case the stability has been justified on finite but asymptotically long time intervals. Now suppose that equation \eqref{TEQ} has three different roots  at the bifurcation point: $\sigma_0$ is a root of multiplicity 2 and $\sigma_1$, $\sigma_2$ are simple roots. Then system \eqref{MS} has two or four autoresonant modes depending on the sign of value  $\mathcal P''(\sigma_0;\delta,\nu)$. In particular, in the case $\mathcal P''(\sigma_0;\delta,\nu)>0$, there are two modes corresponding to the simple roots. In the opposite case, $\mathcal P''(\sigma_0;\delta,\nu)<0$, there are two additional modes corresponding to $\sigma_0$ and associated with the particular solutions having asymptotics \eqref{PAS2}. One of these additional modes with $\psi_1=-\phi$ is stable on the asymptotically long time interval. The asymptotics for the general autoresonant solutions to system \eqref{MS} at the bifurcation points has been described in Theorems~\ref{Thas1} and~\ref{Thas2}.

\begin{figure}
\vspace{-2ex} \centering 
\subfigure[$\nu=0$]{\includegraphics[width=0.3\linewidth]{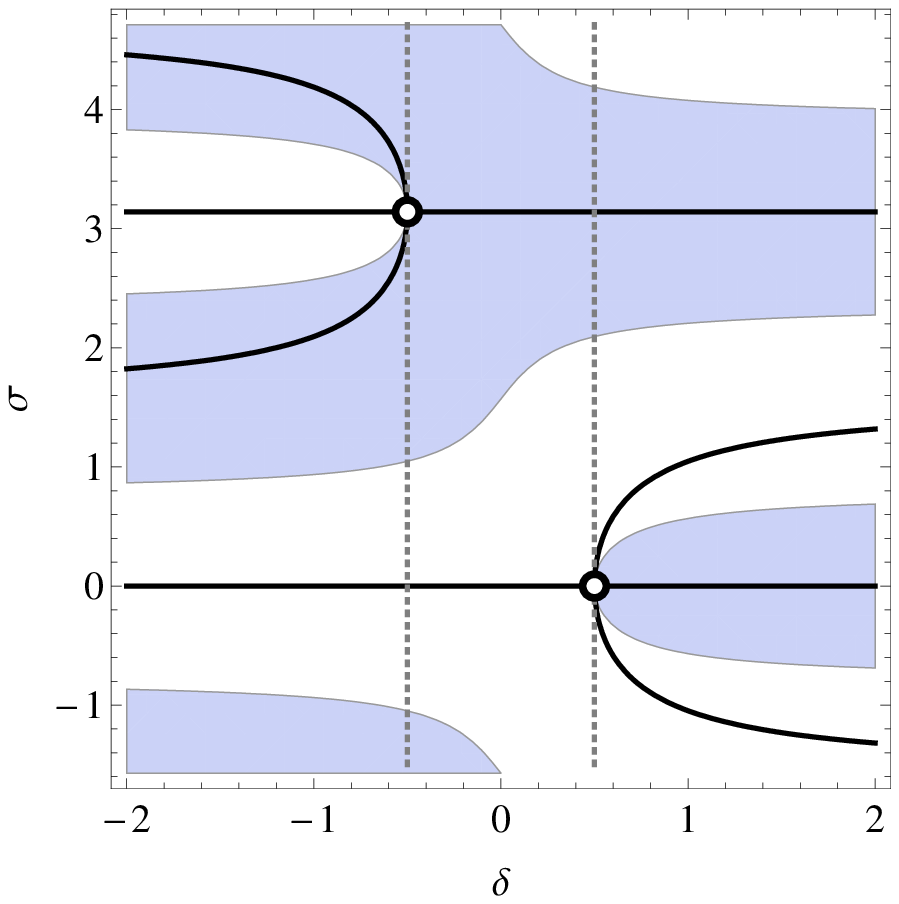}} \hspace{3ex}
\subfigure[$\displaystyle \nu=\frac{\pi}{6}$]{\includegraphics[width=0.3\linewidth]{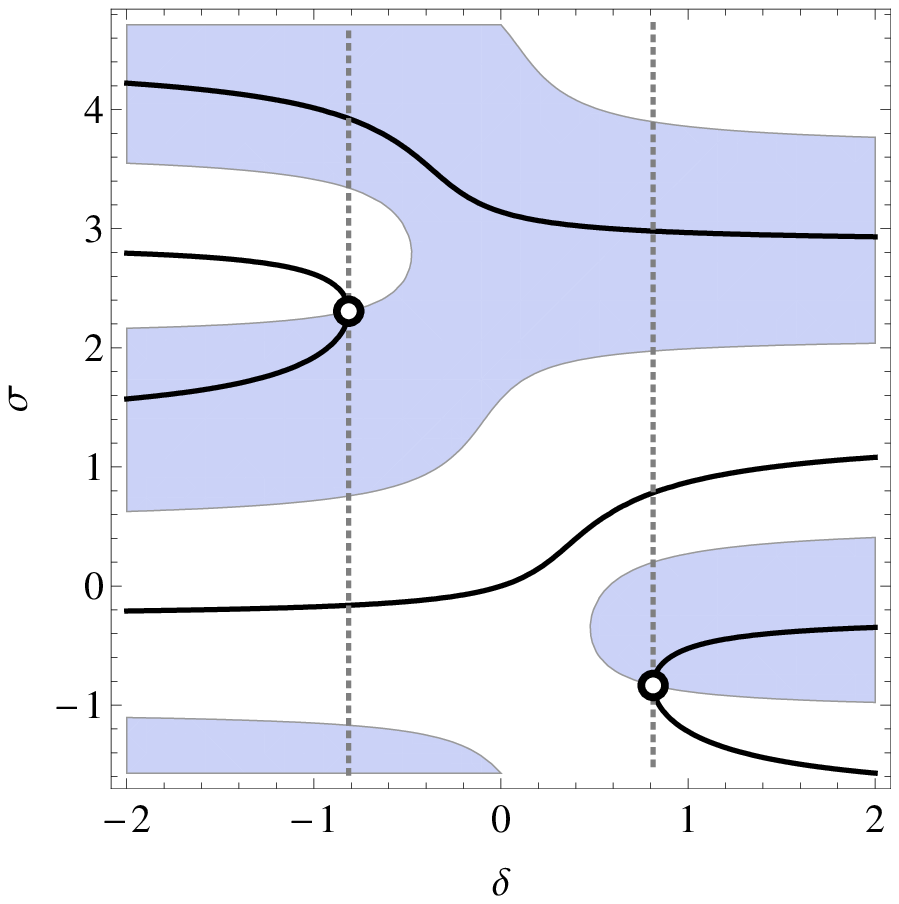}} \hspace{3ex}
\subfigure[$\displaystyle \nu=\frac{5\pi}{6}$]{\includegraphics[width=0.3\linewidth]{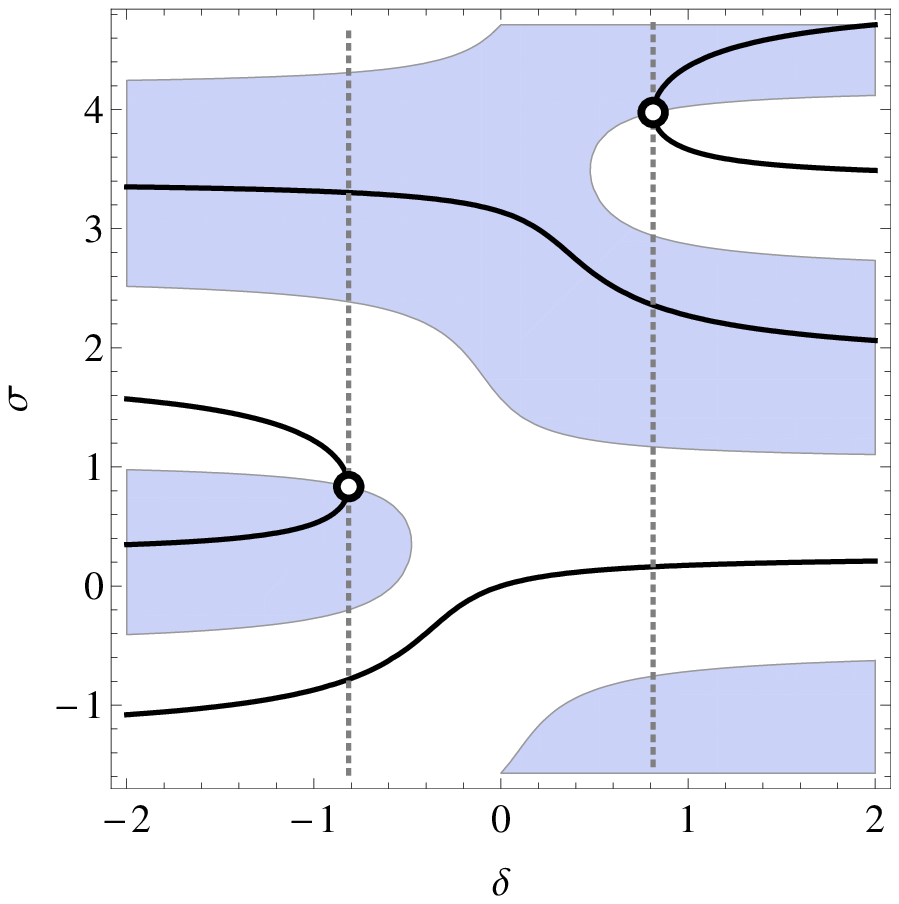}}
\caption{\small The simple roots to equation \eqref{TEQ} as functions of the parameter $\delta$ (black solid lines). The vertical dotted lines correspond to $\gamma_-$ and $\gamma_+$. The shaded areas correspond to $\mathcal P'(\sigma;\delta,\nu)>0$, where the particular solutions to system \eqref{MS} with asymptotics \eqref{PAS} are asymptotically stable.} \label{S0P}
\end{figure}

\begin{figure}
\centering
\subfigure[$\nu=0$]{\includegraphics[width=0.3\linewidth]{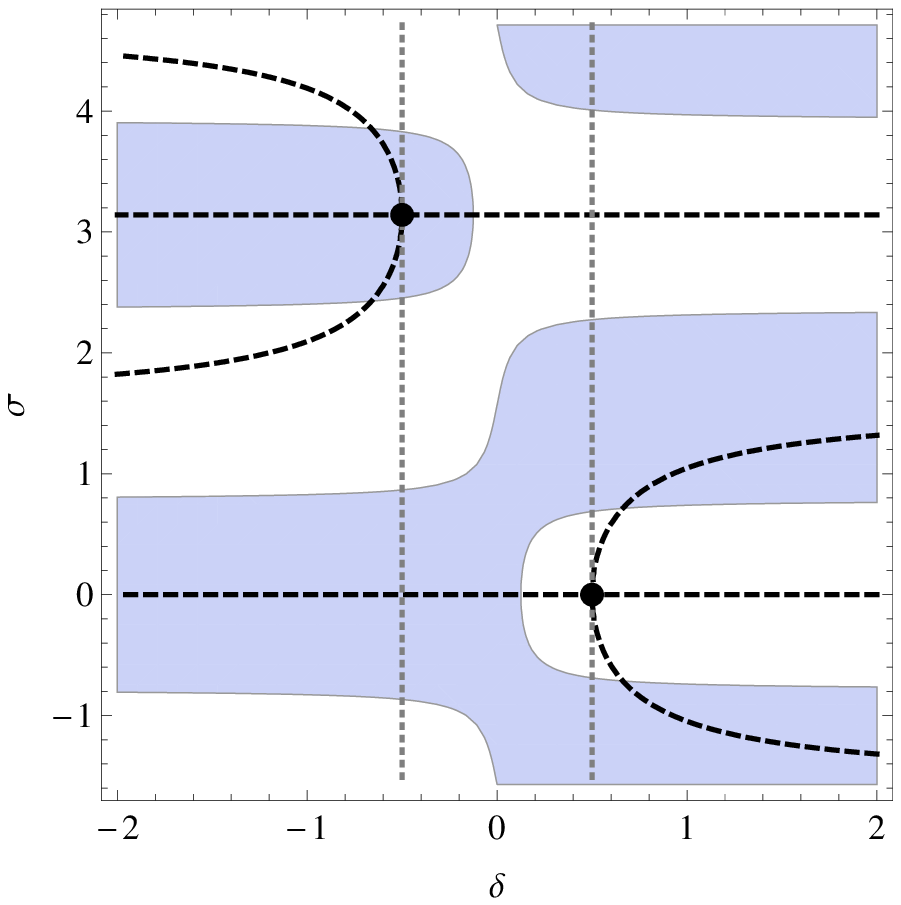}}\hspace{3ex}
\subfigure[$\displaystyle \nu=\frac{\pi}{6}$]{\includegraphics[width=0.3\linewidth]{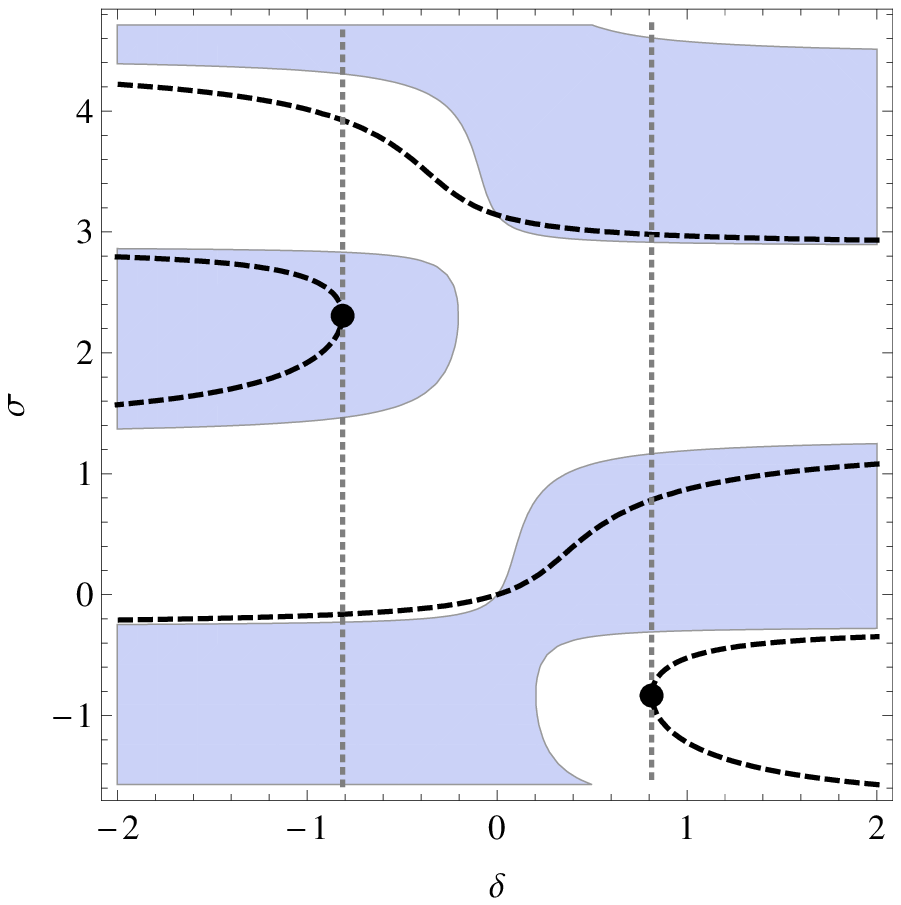}}\hspace{3ex}
\subfigure[$\displaystyle \nu=\frac{5\pi}{6}$]{\includegraphics[width=0.3\linewidth]{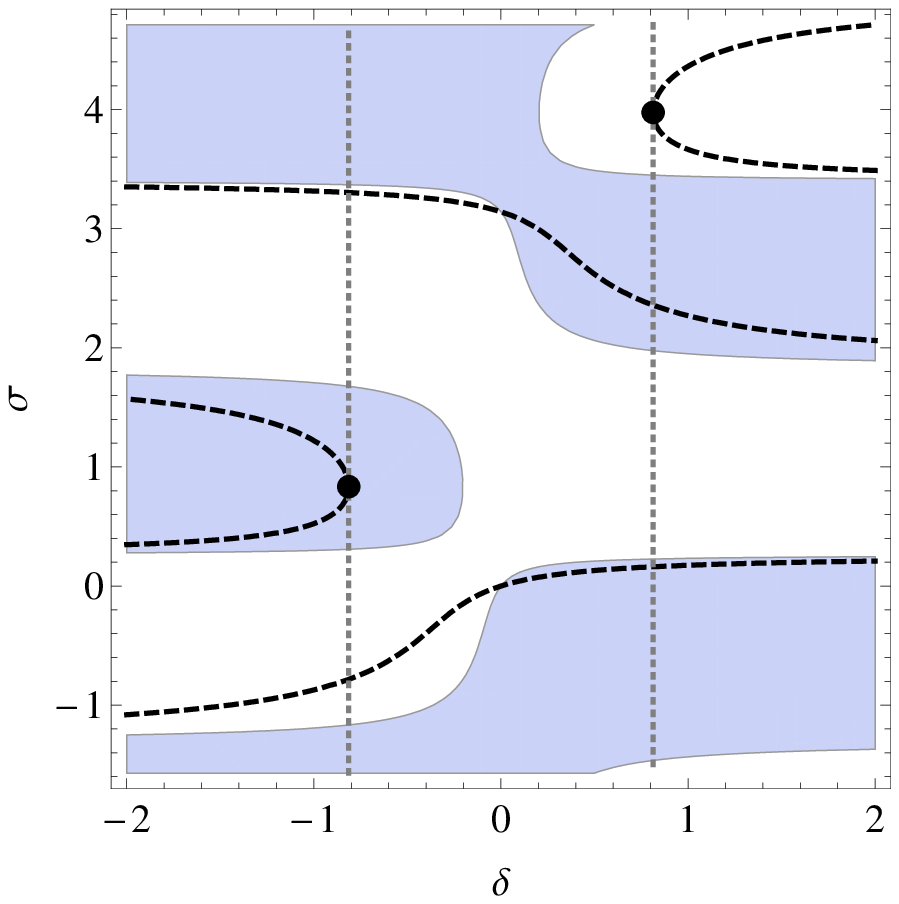}}
\caption{\small The roots to equation \eqref{TEQ} as functions of the parameter $\delta$ (dashed lines). The vertical dotted lines correspond to $\gamma_-$ and $\gamma_+$. The black points correspond to the multiple roots $\sigma_0$. (a) The shaded areas correspond to $\mathcal P'''(\sigma;\delta,\nu)>0$, where the particular solution to system \eqref{MS} with asymptotics \eqref{PAS3} is stable. (b),(c) The shaded areas correspond to $\mathcal P''(\sigma;\delta,\nu)<0$, where two particular solutions to system \eqref{MS} with asymptotics \eqref{PAS2} exist.} \label{M23}
\end{figure}

Note that the stability of the particular autoresonant solutions have been justified by constructing the appropriate Lyapunov functions. These functions have been also used in the asymptotic analysis of the general solutions. In this approach the Lyapunov functions play the role of the action variable that usually appears in the study of the perturbed Hamiltonian systems by the averaging method. The using of the Lyapunov function as a new dependent variable simplifies considerably the construction of the asymptotic solutions to the perturbed non-autonomous systems and reveals the structure of the corresponding averaged equation at once.

\begin{figure}
\centering
\subfigure[$\delta<-\delta_\gamma$]{\includegraphics[width=0.3\linewidth]{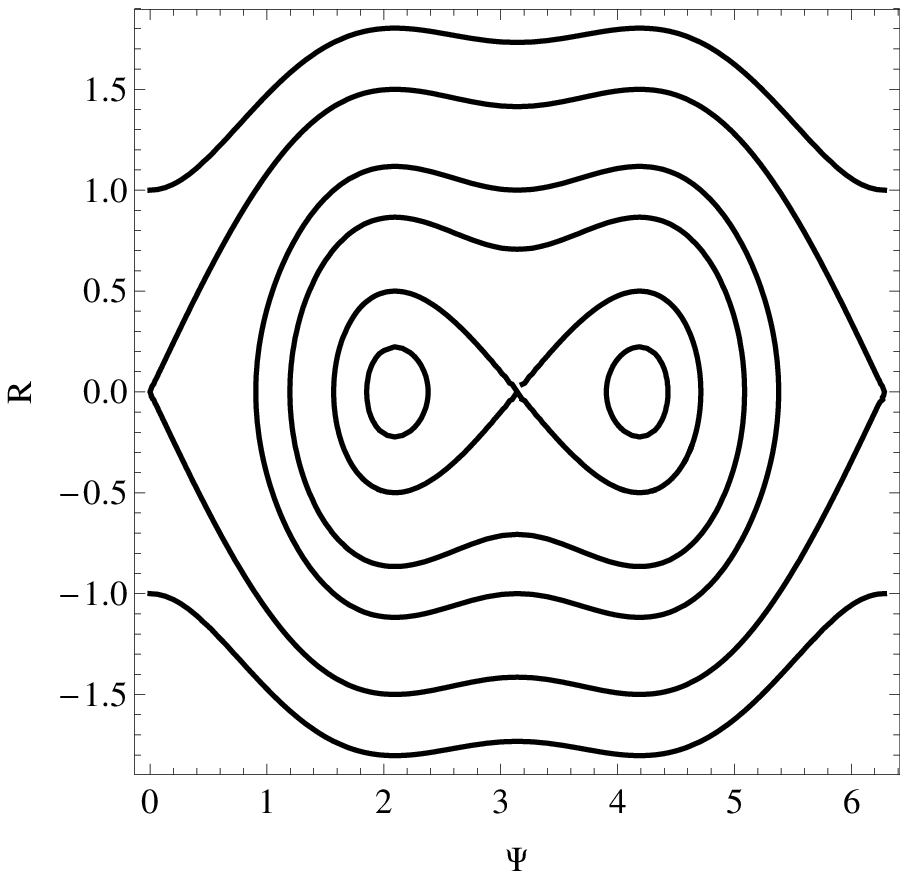}} \hspace{3ex}
\subfigure[$\displaystyle \delta=-\delta_\gamma$]{\includegraphics[width=0.3\linewidth]{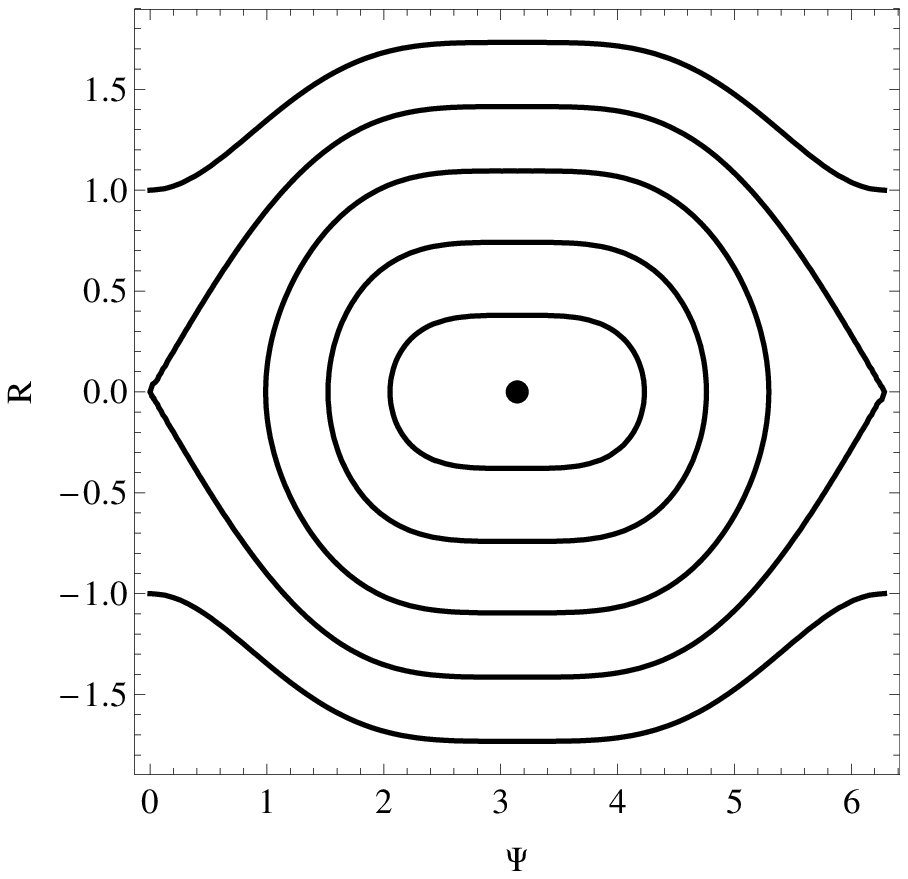}} \hspace{3ex}
\subfigure[$\displaystyle -\delta_\gamma<\delta<0$]{\includegraphics[width=0.3\linewidth]{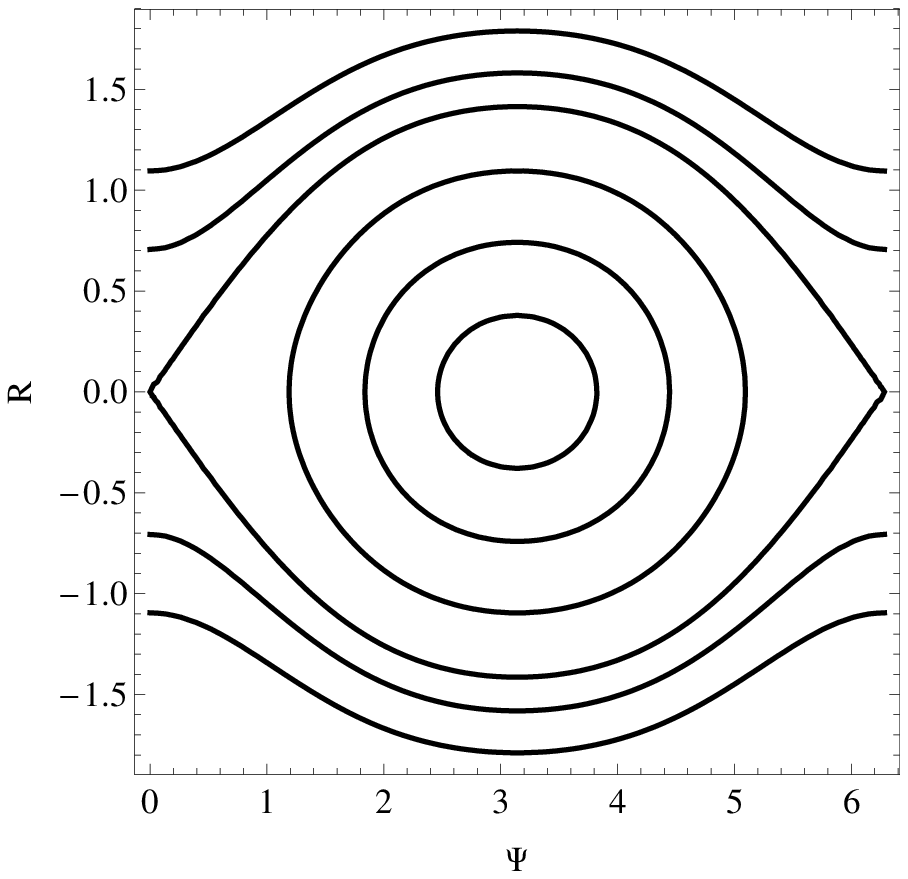}}
\caption{\small Level lines of the Hamiltonian $h_{-1}(R,\Psi)$; $\lambda=1$, $\nu=0$, $  \delta_\gamma={1}/{2}$.  The black point corresponds to $(\sigma_0,0)$, where $\sigma_0$ is the root of multiplicity 3 to equation \eqref{TEQ}.} \label{PP3}
\end{figure}

\begin{figure}
 \centering
\subfigure[$\delta<-\delta_\gamma$]{\includegraphics[width=0.3\linewidth]{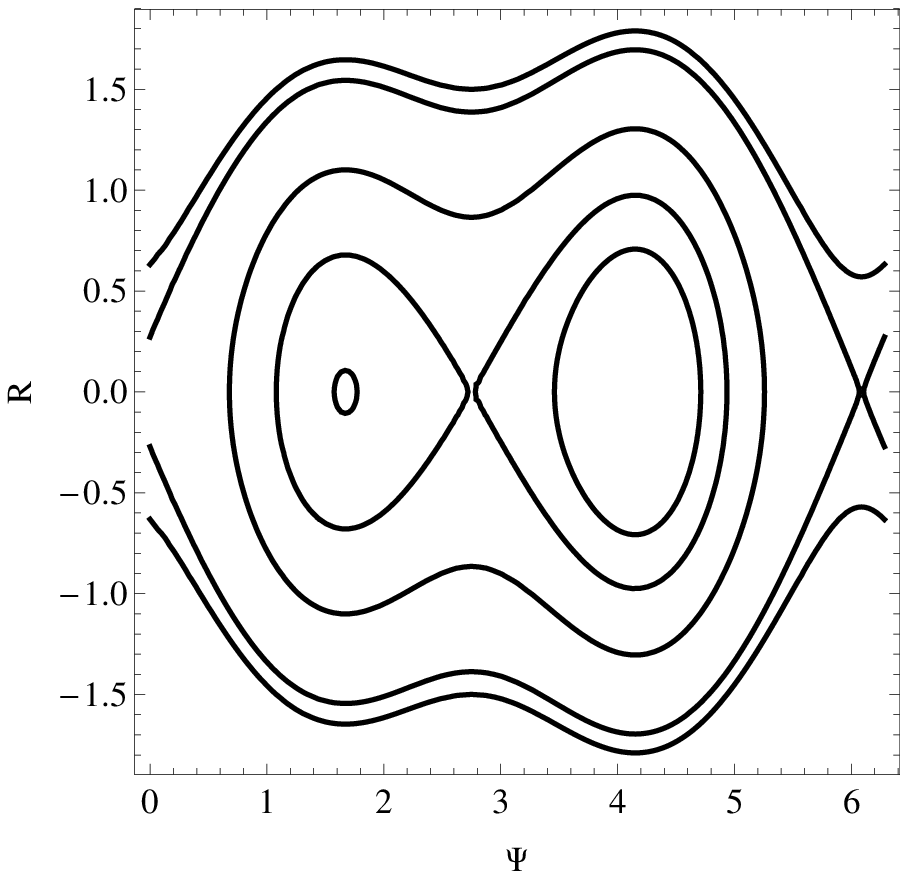}} \hspace{3ex}
\subfigure[$\displaystyle \delta=-\delta_\gamma$]{\includegraphics[width=0.3\linewidth]{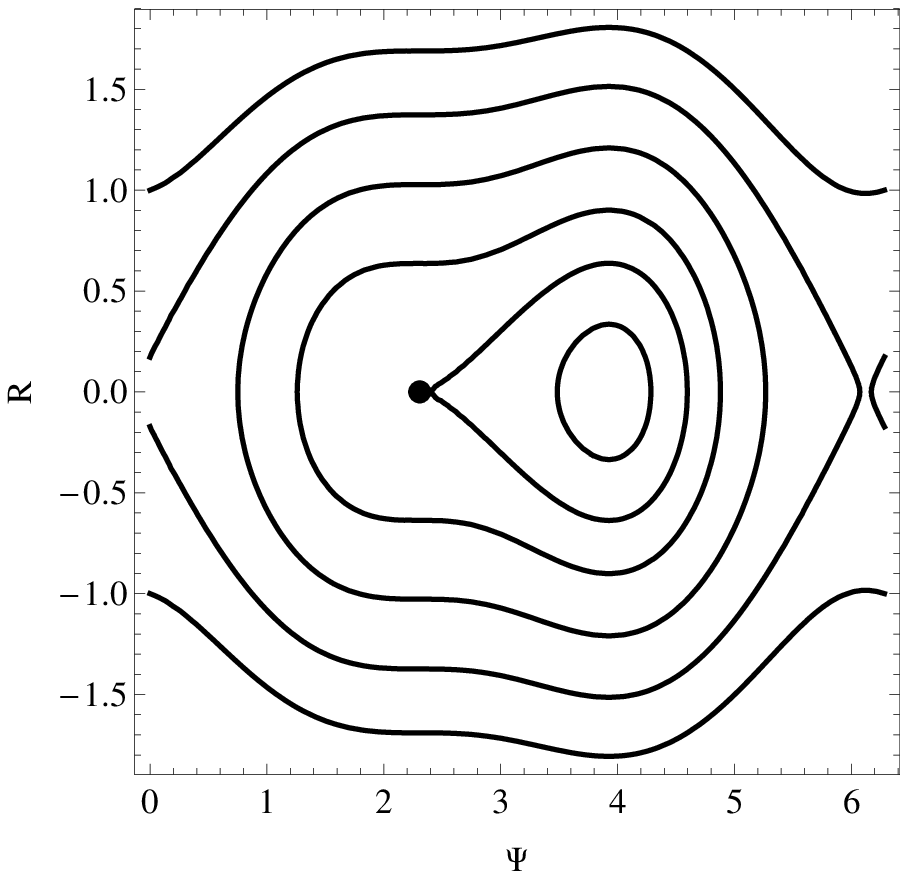}} \hspace{3ex}
\subfigure[$\displaystyle -\delta_\gamma<\delta<0$]{\includegraphics[width=0.3\linewidth]{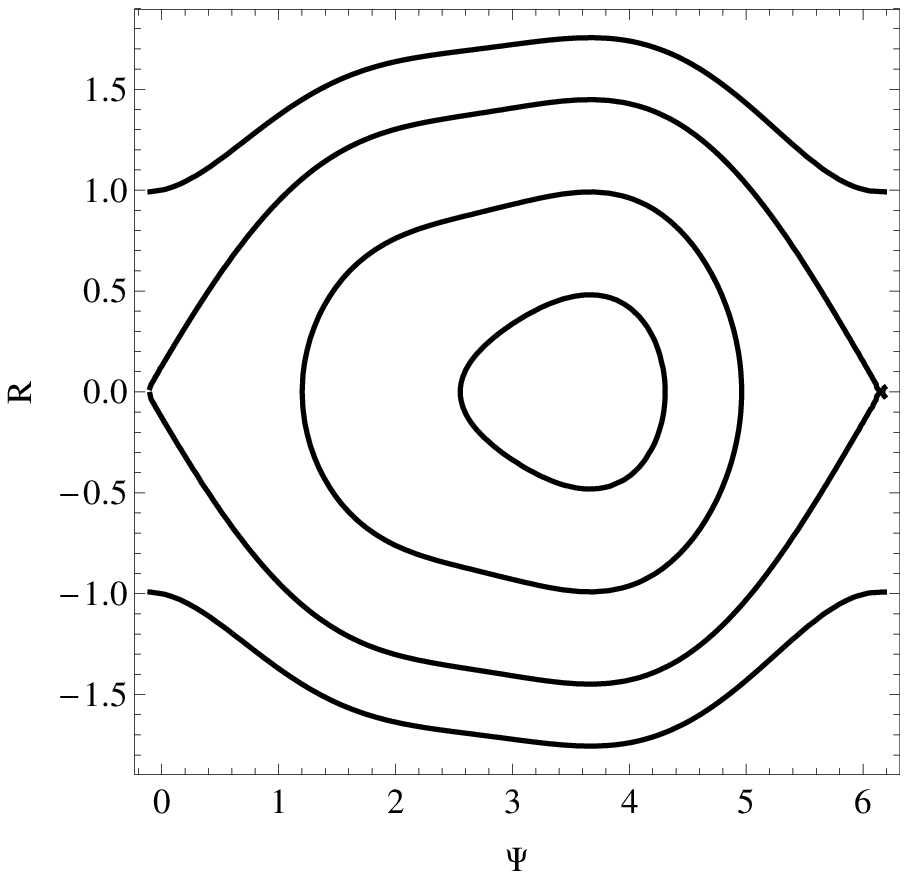}}
\caption{\small Level lines of the Hamiltonian $h_{-1}(R,\Psi)$; $\lambda=1$, $  \nu= {\pi}/{6}$, $\delta_\gamma\approx 0.8134$.  The black point corresponds to $(\sigma_0,0)$, where $\sigma_0$ is the root of multiplicity 2 to equation \eqref{TEQ}.} \label{PP2}
\end{figure}

 It can easily be checked that system \eqref{MS} rewritten in the form of \eqref{ham} corresponds to a weak decaying perturbation of the autonomous system with the Hamiltonian $h_{-}(R,\Psi)$ that, in particular, arises in the study of a simple pendulum with vibrating suspension point. The bifurcations in this system have been discussed, for instance, in~\cite{NSCNSNS17}. When the parameters $(\delta,\nu)$ of the unperturbed Hamiltonian system passes through the curves $\gamma_\pm$, the centre-saddle bifurcation occurs: the centre and the saddle coalesce or the additional pair of points (centre and saddle) appears from the stable or the unstable equilibrium (see Fig.~\ref{PP2} and~\ref{PP3}).  The presence of the time-dependent decaying perturbations in system \eqref{ham} leads to a deformation of the phase portrait near the centres. In particular, some points become asymptotically (polynomially) stable, when the parameters are outside of the bifurcation curves. At the bifurcation points the centres corresponding to multiple roots of equation \eqref{TEQ} become unstable with respect to all the variables. In this case the stability is preserved with respect to one of the variables at least on an asymptotically long time interval.
Note that the bifurcations in non-autonomous systems of differential equations have been considered in several papers using different approaches~\cite{LRS02,KS05,MR08}. However, to the best of our knowledge, the influence of general time-dependent decaying perturbations on the bifurcations of equilibria in Hamiltonian systems has not been thoroughly investigated. This will be discussed elsewhere.

\section*{Acknowledgements}

The author is grateful to L.A. Kalyakin for useful discussions.

}
\end{document}